\documentclass[nosumlimits,twoside]
{amsart}
\usepackage{amsfonts, amsmath, amssymb}
\usepackage[bookmarksnumbered,plainpages,driverfallback=dvipdfm]{hyperref}
\usepackage[english]{babel}
\usepackage{graphicx}
\usepackage{float}
\usepackage{srcltx}
\usepackage[all]{xy}
\usepackage{version}
\usepackage[T1]{fontenc}
\usepackage{xcolor}
\usepackage{enumerate}
\usepackage[normalem]{ulem}
\usepackage{bm}
\usepackage{latexsym}
\usepackage[2emode]{psfrag}
\usepackage{yhmath}
\usepackage{array}
\usepackage{dsfont}
\usepackage{mathrsfs}% for \mathscr
\usepackage{tikz-cd}% diagrams
\usepackage{comment}
\usepackage{mathtools}
\usetikzlibrary{calc,babel}

\usepackage{tikz}
\usepackage{float}

\usetikzlibrary{arrows,arrows.meta,petri,topaths,positioning,shapes,shapes.misc,patterns,calc,decorations.pathreplacing,hobby}

\newtheorem{theorem}{\sc Theorem}[subsection]
\newtheorem{proposition}[theorem]{\sc Proposition}

\newtheorem{lemma}[theorem]{\sc Lemma}
\newtheorem{corollary}[theorem]{\sc Corollary}
\theoremstyle{definition}
\newtheorem{definition}[theorem]{\sc Definition}

\theoremstyle{remark}
\newtheorem{remark}[theorem]{\sc Remark}

\tikzset{
  curve/.style={
    settings={#1},
    to path={
      (\tikztostart)
      .. controls ($(\tikztostart)!\pv{pos}!(\tikztotarget)!\pv{height}!270:(\tikztotarget)$)
      and ($(\tikztostart)!1-\pv{pos}!(\tikztotarget)!\pv{height}!270:(\tikztotarget)$)
      .. (\tikztotarget)\tikztonodes
    },
  },
  settings/.code={%
    \tikzset{quiver/.cd,#1}%
    \def\pv##1{\pgfkeysvalueof{/tikz/quiver/##1}}%
  },
  quiver/.cd,
  pos/.initial=0.35,
  height/.initial=0,
}

\allowdisplaybreaks
\excludeversion{invisible?}
\excludeversion{proof?}

\newcommand{\Id}{{\rm Id}}

\newcommand{\ot}{\otimes}

\newcommand{\Hopfc}{\mathsf{Hopf_{coc}}}
\newcommand{\Hopfcc}{\mathsf{Hopf_{coc}^{com}}}
\newcommand{\HBrc}{\mathsf{HBr_{coc}}}

\DeclareMathOperator{\Hker}{Hker}
\DeclareMathOperator{\Eq}{Eq}

{
\left\{\begin{aligned}}
{\end{aligned}
\right.
}
\title{Central series of cocommutative Hopf braces}
\author{Maria Bevilacqua, Marino Gran, Andrea Sciandra}
\date{}
 \address{%
 \parbox[b]{0.9\linewidth}{Institut de Recherche en Mathématique et Physique, Université Catholique de Louvain, Chemin du Cyclotron 2, B-1348 Louvain-la-Neuve, Belgium.}}
\email{maria.bevilacqua@uclouvain.be}
\email{marino.gran@uclouvain.be}
\address{%
\parbox[b]{0.9\linewidth}{Département de Mathématiques, Université Libre de Bruxelles, Boulevard du Triomphe, B-1050
 Bruxelles, Belgium.}}
\email{andrea.sciandra@ulb.be}


\begin{document}

\begin{abstract}
By extending some classical results known for groups and skew braces, we define and investigate central series of cocommutative Hopf braces. Both left and right central series are defined using a $\star$-product that measures the difference between the two algebra operations, and naturally leads to introducing the notions of socle and of annihilator of a cocommutative Hopf brace. We characterize  the central extensions relative to the subcategories of cocommutative Hopf algebras and of commutative and cocommutative Hopf algebras, respectively. Since the category of cocommutative Hopf braces is semi-abelian and it has enough projectives with respect to the class of cleft extensions, one can then establish suitable Hopf formulae for their homology. These are expressed in terms of the corresponding notions of relative commutators of cocommutative Hopf braces. In particular, the one relative to the subcategory of commutative and cocommutative Hopf algebras turns out to be the Huq commutator. 
\end{abstract}

\maketitle

\section{Introduction}
Skew (left) braces are very interesting algebraic structures that were introduced in \cite{GV} to study solutions of the set-theoretic Yang--Baxter equation (see the book
\cite{CV} for a thorough introduction to the subject). Hopf braces \cite{AGV} can be seen as a natural Hopf-theoretic generalization of skew braces, providing solutions for the quantum Yang--Baxter equation in the cocommutative setting. A Hopf brace is a datum $(H,\cdot,\bullet,1,\Delta,\varepsilon,S,T)$, where $(H,\cdot,1,\Delta,\varepsilon,S)$ and $(H,\bullet,1,\Delta,\varepsilon,T)$ are Hopf algebras satisfying the following compatibility condition
\begin{equation}\label{eq:HBRcomp.}
a\bullet(b\cdot c)=(a_{1}\bullet b)\cdot S(a_{2})\cdot(a_{3}\bullet c), \quad \text{for all}\ a,b,c\in H,
\end{equation} 
where we are employing the Sweedler notation for the comultiplication. A morphism of Hopf braces is a Hopf algebra map with respect to both two Hopf algebra structures. 
The category $\HBrc$ of \emph{cocommutative} Hopf braces, i.e.\ the ones having a cocommutative comultiplication $\Delta$, was recently shown \cite{GranSciandra} to be a semi-abelian  category \cite{JMT}, so that it is natural to investigate commutators, central series, and homology theories in this context. 

In the present article we first show that it is possible to extend the definition of left and right nilpotency from skew braces \cite{CSV} to cocommutative Hopf braces (Definition \ref{central-series}). It turns out that several interesting results can be established in this more general context, many of which depend on some nice properties of the $\star$-product of Hopf braces (see Lemma \ref{lem:propertiesasterisk}, that extends \cite[Lemma $2.7$]{Tsang}). This operation $\star:H\ot H\to H$, defined by 
\[
a\star b:=S(a_{1})\cdot(a_{2}\bullet b_{1})\cdot S(b_{2}), \quad \text{for any}\ a,b \in H
\]
measures the difference between the two multiplications $\cdot$ and $\bullet$ of the Hopf brace $(H,\cdot,\bullet)$.
The existence of the $\star$-product of Hopf braces naturally leads to introducing a definition of the \emph{socle} $\mathrm{Soc}(H)$ and of the \emph{annihilator} $\mathrm{Ann}(H)$ of a cocommutative Hopf brace $(H,\cdot,\bullet)$, that we then prove to be \emph{normal Hopf subbraces} of $(H,\cdot,\bullet)$ (Proposition \ref{NormalHopf}). 

In Section \ref{sec: relative commutator} we look at the adjunction
\begin{equation}\label{Adj HBr Hopf}
\begin{tikzcd}
\HBrc &&& \Hopfc
	\arrow[" "{name=1, anchor=center, inner sep=0}, curve={height=-8pt}, from=1-1, to=1-4,"F"]
	\arrow[" "{name=1, anchor=center, inner sep=0}, curve={height=-8pt}, from=1-4, to=1-1, "G"]
    \arrow["\vdash"{anchor=center, rotate=90}, draw=none, from=1-4, to=1-1]
\end{tikzcd}
\end{equation}
where $\Hopfc$ is the category of cocommutative Hopf algebras, $G\colon\Hopfc\to \HBrc $ is the functor sending any cocommutative Hopf algebra $(H,\cdot,1,\Delta,\varepsilon,S)$ to the ``trivial'' cocommutative Hopf brace $(H,\cdot,\cdot,1,\Delta,\varepsilon,S,S)$
whose two multiplications coincide, and $F\colon \HBrc \to \Hopfc $ is its left adjoint, sending a Hopf brace $(H,\cdot,\bullet)$ to the quotient
\begin{align*}
     F(H) = \frac{H}{H\cdot(H\star H)^+},
\end{align*}
where 
$(H\star H)^+=\{x\in H\star H\ |\ \varepsilon_{H}(x)=0\}$ and $H\star H$ is the Hopf subalgebra of $H^{\cdot}$ generated by elements $a\star b$, for $a,b\in H$. 
We characterize the central extensions of cocommutative Hopf braces relative to this adjunction (Proposition \ref{centralExt}), and this allows us to obtain a Hopf formula for the second homology of a cocommutative Hopf brace (Proposition \ref{prop: H2 Baer invariant}), and to establish a Stallings--Stammbach $5$-term exact sequence for homology (Theorem \ref{thm: Stallings-Stammbach}). The relative commutator $[ \cdot, \cdot]_{\Hopfc}$ that is naturally associated with the adjunction \eqref{Adj HBr Hopf}
yields a notion of nilpotency, with the property that
the $n$-nilpotent Hopf braces form a Birkhoff subcategory $\mathsf{Nil}^{(n)} $ of $\HBrc$ (for any positive integer $n$, see Proposition \ref{Birkhoff}). The notion of nilpotency arising in this way is closely related to the natural Hopf-theoretic version of the \emph{right nilpotency} of skew braces and of their right central series, that have already been studied in the literature (see \cite{Tsang}, and the references therein). More precisely, a cocommutative Hopf brace $(H,\cdot,\bullet)$ belongs to $\mathsf{Nil}^{(n)}$ if and only if $[H^{(n)},H]_{\Hopfc}= \Bbbk 1_H$ , where $H^{(n)}$ is the $n$-th term of the right series of $H$ (Definition \ref{central-series}); this means that both $H^{(n)}\star H$ and $H\star H^{(n)}$ are trivial, whereas the right nilpotency of class $n$ only requires the triviality of $H^{(n+1)}=H^{(n)}\star H$. Accordingly, any Hopf brace in $\mathsf{Nil}^{(n)}$ is right nilpotent of class at most $n$, and the two notions coincide when $n=1$, in which case they both amount to  ${H\star H = \Bbbk}1_H$, i.e.\ to $(H,\cdot,\bullet)$ being a cocommutative Hopf algebra.

%It turns out that the notions of nilpotency and of central series arising in this way coincide with the natural Hopf-theoretic version of the \emph{right nilpotency} of skew braces and their central series, that have already been studied in the literature (see \cite{Tsang}, and the references therein). 

In Section \ref{abelian} we look at the adjunction 
\begin{equation}\label{adj Hbr Hopfcc}
\begin{tikzcd}
\HBrc &&& \mathsf{Hopf}_{\mathsf{coc}}^{\mathsf{com}}
	\arrow[" "{name=1, anchor=center, inner sep=0}, curve={height=-8pt}, from=1-1, to=1-4,"\mathsf{ab}"]
	\arrow[" "{name=1, anchor=center, inner sep=0}, curve={height=-8pt}, from=1-4, to=1-1, "U"]
    \arrow["\vdash"{anchor=center, rotate=90}, draw=none, from=1-4, to=1-1]
\end{tikzcd}
\end{equation}
\noindent where $U \colon \mathsf{Hopf}_{\mathsf{coc}}^{\mathsf{com}} \rightarrow \HBrc $ is the functor sending any commutative and cocommutative Hopf algebra to the cocommutative Hopf brace whose two multiplications coincide and are commutative, while the left adjoint $ \mathsf{ab} \colon \HBrc \rightarrow \mathsf{Hopf}_{\mathsf{coc}}^{\mathsf{com}}  $ 
is the abelianisation functor in the semi-abelian category $\HBrc$, that has been explicitly described in \cite[Corollary 8.8]{GranSciandra}. Indeed, the full subcategory $\mathsf{Hopf}_{\mathsf{coc}}^{\mathsf{com}} $ of commutative and cocommutative Hopf algebras coincides with the category of \emph{abelian objects} in $\HBrc$. The notions of commutator and of central extension corresponding to this second adjunction \eqref{adj Hbr Hopfcc} agree with the Huq commutator of normal subobjects \cite{Huq} and with the categorical notion of central extension \cite{BournGran}. These give rise to a different notion of central series (Proposition \ref{prop: Gamma_n normal}), that in the special case of skew braces corresponds to the ones considered in \cite{BP, KanrarRoelantsYadav}.

%\section{Preliminaries}*
%We recall some notions and results that will be useful throughout the paper.
%\begin{definition}[{\cite[Definition 1.1]{AGV}}]
%A \textit{Hopf brace} is a datum $(H,\cdot,\bullet,1,\Delta,\varepsilon,S,T)$ where $(H,\cdot,1,\Delta,\varepsilon,S)$ and $(H,\bullet,1,\Delta,\varepsilon,T)$ are Hopf algebras satisfying the following compatibility condition:
%\begin{equation}\label{eq:HBRcomp.}
%a\bullet(b\cdot c)=(a_{1}\bullet b)\cdot S(a_{2})\cdot(a_{3}\bullet c), \quad \text{for all}\ a,b,c\in H.
%\end{equation}
%A morphism of Hopf braces $f:(H,\cdot,\bullet)\to(K,\cdot,\bullet)$ is a Hopf algebra morphism with respect to both the Hopf algebra structures.
% , i.e. it satisfies
% \[
% f(a\cdot b)=f(a)\cdot f(b),\quad f(a\bullet b)=f(a)\bullet f(b),\quad f(1_{H})=1_{K}
% \]
% and
% \[
% \varepsilon(f(a))=\varepsilon(a),\quad \Delta(f(a))=f(a_{1})\otimes f(a_{2}), 
% \]
% for all $a,b\in H$. Any morphism of Hopf braces automatically preserves the antipodes, so that $f S_{H}=S_{K}f$ and $fT_{H}=T_{K}f$. 
%A Hopf brace is \textit{cocommutative} if the comultiplication $\Delta$ is cocommutative. We denote the category of cocommutative Hopf braces by $\HBrc$.
%\end{definition}
% Hopf braces provide a Hopf-theoretic generalization of \textit{skew braces}, introduced in \cite{GV}: these are objects $(G,\cdot,\bullet)$, where $(G,\cdot)$ and $(G,\bullet)$ are groups and, for all $a,b,c\in G$, the following equality is satisfied:
% $a\bullet(b\cdot c)=(a\bullet b)\cdot a^{-1}\cdot(a\bullet c)$, where $a^{-1}$ denotes the inverse of $a$ in $(G,\cdot)$.

\section{Central series for cocommutative Hopf braces}

We will usually denote a cocommutative Hopf brace simply by $(H,\cdot,\bullet)$, and adopt the notations $H^{\cdot}:=(H,\cdot,1,\Delta,\varepsilon,S)$ and $H^{\bullet}:=(H,\bullet,1,\Delta,\varepsilon,T)$ for the two Hopf algebra structures. Since $H$ is cocommutative, both antipodes $S$ and $T$ are involutive, i.e.\ $S^{2}=\Id_{H}=T^{2}$; this fact will be used freely in the computations throughout the paper.

The Hopf algebra $H^{\cdot}$ of a Hopf brace $(H,\cdot,\bullet)$ is a left $H^{\bullet}$-module algebra (see \cite{AGV}), where the left $H^{\bullet}$-action is defined by:
\begin{equation}
\label{def:leftaction}
a\rightharpoonup b:=S(a_{1})\cdot(a_{2}\bullet b),\quad \text{for all }a,b\in H.
\end{equation}
In particular, the following equations are satisfied, for all $a,b,c\in H$:
\[
a\rightharpoonup1=\varepsilon(a)1,
\qquad a\rightharpoonup(b\cdot c)=(a_{1}\rightharpoonup b)\cdot(a_{2}\rightharpoonup c).
\]
Moreover, $H^{\cdot}$ becomes a Hopf algebra in the category of left $H^{\bullet}$-modules (with braiding given by the flip map), so the following equalities hold, for all $a,b\in H$:
\[
\Delta(a\rightharpoonup b)=(a_{1}\rightharpoonup b_{1})\ot(a_{2}\rightharpoonup b_{2}),\quad \varepsilon(a\rightharpoonup b)=\varepsilon(a)\varepsilon(b),\quad S( a \rightharpoonup b) = a \rightharpoonup S(b).
\]
The definition of Hopf brace also implies the following equalities (see \cite{AGV})
\[
a\bullet b=a_{1}\cdot(a_{2}\rightharpoonup b),\quad a\cdot b=a_{1}\bullet(T(a_{2})\rightharpoonup b),\quad S(a)=a_{1}\rightharpoonup T(a_{2}).
\]
By extending the notion of \textit{left ideal} of a skew brace, we introduce:

\begin{definition}
    Let $(H,\cdot,\bullet)$ be a cocommutative Hopf brace. A vector subspace $B\subseteq H$ is a \textit{Hopf strong subbrace} if it is a Hopf subalgebra $B^{\cdot}$ of $H^{\cdot}$ and $h\rightharpoonup b\in B$, for all $h\in H,b\in B$.
\end{definition}

\begin{remark}\label{rmk:Hopfsubbullet}
Since $B^{\cdot}$ is a Hopf subalgebra of $H^{\cdot}$ and $H\rightharpoonup B\subseteq B$, given $b,b'\in B$, one has $b\bullet b'=b_{1}\cdot(b_{2}\rightharpoonup b')\in B$ and, since $H$ is cocommutative, we have 
\[
T(b)=SST(b)=S(ST(b_{1})\cdot(T(b_{2})\bullet b_{3}))=S(T(b_{1})\rightharpoonup b_{2})\in B.
\]
Hence, $B^{\bullet}$ is a Hopf subalgebra of $H^{\bullet}$ and so $(B,\cdot,\bullet)$ is a Hopf subbrace of $(H,\cdot,\bullet)$. 
\end{remark}

We define the operation $\star:H\ot H\to H$ in the following way:
\begin{equation}\label{def:star}
a\star b:=S(a_{1})\cdot(a_{2}\bullet b_{1})\cdot S(b_{2})=(a\rightharpoonup b_{1})\cdot S(b_{2}),\quad \text{for any}\ a,b \in H.
\end{equation}

This operation was introduced and first studied in \cite[Section 2]{KSV} for skew braces, as a tool that measures the difference between the two underlying group operations $\cdot$ and $\bullet$ of a skew brace. By analogy, we call it \textit{asterisk} or \textit{star product}. 

\begin{remark}\label{rmk:closureaction}
    Given a cocommutative Hopf brace $(H,\cdot,\bullet)$ and a Hopf subalgebra $B^{\cdot}$ of $H^{\cdot}$, for any $a\in H$, we have that $a\rightharpoonup b\in B$ for all $b\in B$ if and only if $a\star b=(a\rightharpoonup b_{1})\cdot S(b_{2})\in B$ for all $b\in B$. We also observe that $\Delta(a\star b)=(a_{1}\star b_{1})\ot(a_{2}\star b_{2})$ and $\varepsilon(a\star b)=\varepsilon(a)\varepsilon(b)$, i.e.\ $\star$ is a morphism of coalgebras.
\end{remark}

First, we prove the following result, that generalizes \cite[Lemma 2.7]{Tsang}.

\begin{lemma}\label{lem:propertiesasterisk}
    Let $(H,\cdot,\bullet)$ be a cocommutative Hopf brace. For any $a,x,y\in H$, we have:
\begin{itemize}
    \item[1)] $a\star(x\cdot y)= (a_{1}\star x_{1})\cdot x_{2}\cdot(a_{2}\star y)\cdot S(x_{3})$,
    \item[2)] $x\star(y\star a)=(x_1\bullet y_1\rightharpoonup a_1)\cdot (x_2\rightharpoonup S(a_2))\cdot a_3\cdot (y_2\rightharpoonup S(a_4)),$
    \item[3)] $(x\bullet y)\star a=(x_{1}\star(y_{1}\star a_{1}))\cdot(y_{2}\star a_{2})\cdot(x_{2}\star a_{3})$,
    \item[4)] $a\rightharpoonup(x\star y)=(a_{1}\bullet x\bullet T(a_{2}))\star (a_{3}\rightharpoonup y)$,
    \item[5)] $a_{1}\bullet x\bullet T(a_{2})=a_{1}\cdot(a_{2}\rightharpoonup(x_{1}\cdot(x_{2}\star T(a_{3}))))\cdot S(a_{4})$.
\end{itemize}
\end{lemma}

\begin{proof}
    We compute
\begin{align*}
a\star(x\cdot y)&=(a\rightharpoonup(x\cdot y)_{1})\cdot S((x\cdot y)_{2})=(a\rightharpoonup x_{1}\cdot y_{1})\cdot S(x_{2}\cdot y_{2})\\&=(a_{1}\rightharpoonup x_{1})\cdot(a_{2}\rightharpoonup y_{1})\cdot S(y_{2})\cdot S(x_{2})\\&=(a_{1}\rightharpoonup x_{1})\cdot S(x_{2})\cdot x_{3}\cdot(a_{2}\rightharpoonup y_{1})\cdot S(y_{2})\cdot S(x_{4})\\&=(a_{1}\star x_{1})\cdot x_{2}\cdot(a_{2}\star y)\cdot S(x_{3}),
\end{align*}
and also
\begin{align*}
    x\star (y\star a)&=(x\star ((y\rightharpoonup a_1)\cdot S(a_{2})))\\&=(x\rightharpoonup ((y_{1}\rightharpoonup a_{1})\cdot S(a_{2})))\cdot S((y_{2}\rightharpoonup a_{3})\cdot S(a_{4}))\\&=(x_1\bullet y_{1}\rightharpoonup a_{1})\cdot(x_{2}\rightharpoonup S(a_{2}))\cdot a_{3}\cdot S(y_{2}\rightharpoonup a_{4})\\&=
    (x_1\bullet y_{1}\rightharpoonup a_{1})\cdot(x_{2}\rightharpoonup S(a_{2}))\cdot a_{3}\cdot (y_2\rightharpoonup S(a_4)),
\end{align*}
proving that 1) and 2) are satisfied. Moreover, we have
\begin{align*}
    &(x_{1}\star (y_{1}\star a_{1}))\cdot(y_{2}\star a_{2})\cdot(x_{2}\star a_{3})\\&\overset{2)}{=}(x_{1}\bullet y_{1}\rightharpoonup a_{1})\cdot(x_{2}\rightharpoonup S(a_{2}))\cdot a_{3}\cdot S(y_{2}\rightharpoonup a_{4})\cdot (y_{3}\rightharpoonup a_{5})\cdot S(a_{6})\cdot(x_{3}\rightharpoonup a_{7})\cdot S(a_{8})\\&=(x_{1}\bullet y\rightharpoonup a_{1})\cdot(x_{2}\rightharpoonup S(a_{2}))\cdot a_{3}\cdot S(a_{4})\cdot(x_{3}\rightharpoonup a_{5})\cdot S(a_{6})\\&=(x_{1}\bullet y\rightharpoonup a_{1})\cdot(x_{2}\rightharpoonup S(a_{2}))\cdot (x_{3}\rightharpoonup a_{3})\cdot S(a_{4})\\&=(x_{1}\bullet y\rightharpoonup a_{1})\cdot(x_{2}\rightharpoonup S(a_{2})\cdot a_{3})\cdot S(a_{4})\\&=(x\bullet y\rightharpoonup a_{1})\cdot S(a_{2})=(x\bullet y)\star a,
\end{align*}
which proves 3). Finally, in order to show 4) and 5), we have to compute the following:
\begin{align*}
a\rightharpoonup(x\star y)&=a\rightharpoonup((x\rightharpoonup y_{1})\cdot S(y_{2}))= (a_{1}\rightharpoonup(x\rightharpoonup y_{1}))\cdot(a_{2}\rightharpoonup S(y_{2}))\\&=(a_{1}\bullet x\rightharpoonup y_{1})\cdot(a_{2}\rightharpoonup S(y_{2}))=(a_{1}\bullet x\bullet T(a_{2})\rightharpoonup(a_{3}\rightharpoonup y_{1}))\cdot S(a_{4}\rightharpoonup y_{2})\\&=(a_{1}\bullet x\bullet T(a_{2}))\star(a_{3}\rightharpoonup y)
\end{align*}
and also
\begin{align*}
&a_{1}\cdot(a_{2}\rightharpoonup(x_{1}\cdot(x_{2}\star T(a_{3}))))\cdot S(a_{4})\\&=a_{1}\cdot(a_{2}\rightharpoonup x_{1})\cdot(a_{3}\rightharpoonup x_{2}\star T(a_{4}))\cdot S(a_{5})\\&\overset{4)}{=} a_{1}\cdot(a_{2}\rightharpoonup x_{1})\cdot((a_{3}\bullet x_{2}\bullet T(a_{4}))\star (a_{5}\rightharpoonup T(a_{6})))\cdot S(a_{7})\\&=a_{1}\cdot(a_{2}\rightharpoonup x_{1})\cdot(a_{3}\bullet x_{2}\bullet T(a_{4})\rightharpoonup(a_{5}\rightharpoonup T(a_{6})))\cdot S(a_{7}\rightharpoonup T(a_{8}))\cdot S(a_{9})\\&=(a_{1}\bullet x_{1})\cdot(a_{2}\bullet x_{2}\rightharpoonup T(a_{3}))\cdot S(a_{4}\cdot(a_{5}\rightharpoonup T(a_{6})))\\&=(a_{1}\bullet x)_{1}\cdot((a_{1}\bullet x)_{2}\rightharpoonup T(a_{2}))\cdot S(a_{3}\bullet T(a_{4}))\\&=a_{1}\bullet x\bullet T(a_{2}),
\end{align*}
which concludes the proof.
\end{proof}

It is known that $\star$ is not associative in the case of skew braces, thus the same can be said at the level of cocommutative Hopf braces. This is evident from 2) of Lemma \ref{lem:propertiesasterisk}.

We recall the following definition:
\begin{definition}[{\cite[Definition 4.4]{GranSciandra}}]\label{def:normal}
A Hopf subbrace $(B,\cdot,\bullet)$ of a Hopf brace $(H,\cdot,\bullet)$ is \textit{normal} if, for all $a\in H,b\in B$, the following are satisfied:
\begin{itemize}
    \item[1)] $a_{1}\cdot b\cdot S(a_{2})\in B$, i.e., $B^{\cdot}$ is a normal Hopf subalgebra of $H^{\cdot}$,
    \item[2)] $a_{1}\bullet b\bullet T(a_{2})\in B$, i.e., $B^{\bullet}$ is a normal Hopf subalgebra of $H^{\bullet}$,
    \item[3)] $a\rightharpoonup b\in B$, where $\rightharpoonup$ is defined as in \eqref{def:leftaction}.
\end{itemize}
\end{definition}

\begin{remark}
As proven in \cite[Lemma 4.3]{GranSciandra}, the condition $H\rightharpoonup B\subseteq B$ implies that $H\cdot B^{+}=H\bullet B^{+}$, where $B^{+}=\{x\in B\ |\ \varepsilon(x)=0\}$, and then the fact that $(B,\cdot,\bullet)$ is a normal Hopf subbrace of $(H,\cdot,\bullet)$ is equivalent to $I:=H\cdot B^{+}=H\bullet B^{+}$ being a Hopf brace ideal and to the fact that $(B,\cdot,\bullet)$ is the categorical kernel in $\HBrc$ of the canonical quotient $\pi:(H,\cdot,\bullet)\to(H/I,\overline{\cdot},\overline{\bullet})$, as shown in \cite[Proposition 4.6]{GranSciandra}. 
\end{remark}

The following result provides an equivalent characterization for normal Hopf subbraces using the properties of the map $\star$.

\begin{corollary}\label{cor:equivalentnotionsnormal}
    Let $(H,\cdot,\bullet)$ be a cocommutative Hopf brace. Then $(B,\cdot,\bullet)$ is a normal Hopf subbrace of $(H,\cdot,\bullet)$ if and only if $B^{\cdot}$ is a normal Hopf subalgebra of $H^{\cdot}$, $b\star a\in B$ and $a\star b\in B$ for all $a\in H$ and $b\in B$.
\end{corollary}

\begin{proof}
Suppose that $(B,\cdot,\bullet)$ is a normal Hopf subbrace of $(H,\cdot,\bullet)$. In particular, $B^{\cdot}$ is a normal Hopf subalgebra of $H^{\cdot}$ and, by Remark \ref{rmk:closureaction}, $a\star b\in B$ for all $a\in H$, $b\in B$. Moreover, by 5) of Lemma \ref{lem:propertiesasterisk}, we have
    \[
    b\star T(a)=S(b_{1})\cdot\Big(T(a_{1})\rightharpoonup\big(S(a_{2})\cdot(a_{3}\bullet b_{2}\bullet T(a_{4}))\cdot a_{5}\big)\Big)\in B
    \]
since $B^{\bullet}$ is a normal Hopf subalgebra of $H^{\bullet}$, $B^{\cdot}$ is a normal Hopf subalgebra of $H^{\cdot}$ and $H\rightharpoonup B\subseteq B$. Then, since $T$ is bijective, $b\star a\in B$ for all $a\in H$ and $b\in B$. 
    
Now, suppose that $B^{\cdot}$ is a normal Hopf subalgebra of $H^{\cdot}$, $b\star a\in B$ and $a\star b\in B$ for all $a\in H$, $b\in B$. By Remark \ref{rmk:closureaction}, we have $H\rightharpoonup B\subseteq B$. Then, by Remark \ref{rmk:Hopfsubbullet}, $B$ is a Hopf subalgebra of $H^{\bullet}$. It is also normal since, by 5) of Lemma \ref{lem:propertiesasterisk}, we have
\[
a_{1}\bullet b\bullet T(a_{2})=a_{1}\cdot(a_{2}\rightharpoonup(b_{1}\cdot(b_{2}\star T(a_{3}))))\cdot S(a_{4})\in B
\]
since $b\star a\in B$ for all $b\in B$ and $a\in H$, $H\rightharpoonup B\subseteq B$ and $B^{\cdot}$ is normal in $H^{\cdot}$.
\end{proof}

Since $\star$ is not associative, it matters whether we $\star$-multiply $H$ on the left or on the right. For any vector subspaces $X,Y$ of $H$, we define $X\star Y$ to be the Hopf subalgebra of $H^{\cdot}$ generated by the elements $x\star y$, for $x\in X$ and $y\in Y$. This is defined as the smallest Hopf subalgebra of $H^{\cdot}$ containing the vector space $\{x\star y\ |\ x\in X,y\in Y\}$.

\begin{remark}
    If $X,Y$ are subcoalgebras of $(H,\Delta,\varepsilon)$, from $\Delta(a\star b)=(a_{1}\star b_{1})\ot(a_{2}\star b_{2})$ we immediately have that $\{x\star y\ |\ x\in X,y\in Y\}$ is a subcoalgebra of $(H,\Delta,\varepsilon)$.
\end{remark}

We introduce the following definition:

\begin{definition}\label{central-series}
    Let $(H,\cdot,\bullet)$ be a cocommutative Hopf brace. The \textit{left series} of $(H,\cdot,\bullet)$ is defined by $H^{1}:=H$, $H^{n+1}:=H\star H^{n}$, for $n\geq1$, and the \textit{right series} of $(H,\cdot,\bullet)$ is defined by $H^{(1)}:=H$, $H^{(n+1)}:=H^{(n)}\star H$, for $n\geq1$.
\end{definition}

In analogy with the corresponding results for skew braces \cite[Propositions 2.1 and 2.2]{CSV}, we obtain:

\begin{proposition}
    For all $n\geq 1$, $H^{n}$ is a Hopf strong subbrace of $(H,\cdot,\bullet)$.
\end{proposition}

\begin{proof}
We only have to verify that $H\rightharpoonup H^{n}\subseteq H^{n}$, for all $n\geq1$. We prove this by induction. Of course, this is true for $n=1$. Suppose that $H\rightharpoonup H^{n}\subseteq H^{n}$, for an arbitrary $n$. For any $a,b\in H$ and $x\in H^{n}$, %we have $a\rightharpoonup x\in H^{n}$ and so, 
by 4) of Lemma \ref{lem:propertiesasterisk}, we have 
\[
    a\rightharpoonup(b\star x)=(a_{1}\bullet b\bullet T(a_{2}))\star(a_{3}\rightharpoonup x)\in H\star H^{n}=H^{n+1}. 
\]
Since $H^{n+1}$ is the Hopf subalgebra of $H^{\cdot}$ generated by elements of the form $b\star x$, with $b\in H$ and $x\in H^{n}$, and $a\rightharpoonup p\cdot q=(a_{1}\rightharpoonup p)\cdot(a_{2}\rightharpoonup q)$ for all $a,p,q\in H$, we obtain $H\rightharpoonup H^{n+1}\subseteq H^{n+1}$.
\end{proof}

\begin{proposition}\label{prop:A^(n) normal in A}
    For all $n\geq1$, $H^{(n)}$ is a normal Hopf subbrace of $(H,\cdot,\bullet)$.
\end{proposition}

\begin{proof}
For all $n\geq1$, we have to verify that $H\rightharpoonup H^{(n)}\subseteq H^{(n)}$, from which it follows that $H^{(n)}$ is also a Hopf subalgebra of $H^{\bullet}$ (see Remark \ref{rmk:Hopfsubbullet}), and the fact that $H^{(n)}$ is normal in $H^{\cdot}$ and in $H^{\bullet}$. We prove these properties by induction. Of course, they are true for $n=1$. Suppose now these conditions hold for an arbitrary $n$. We first observe that, given $a\in H^{(n)}$ and $b\in H$, $a\star b\in H^{(n)}$
% \[
% \begin{split}
% a\star b&=S(a_{1})\cdot(a_{2}\bullet b_{1})\cdot S(b_{2})=S(a_{1})\cdot(b_{1}\bullet T(b_{2})\bullet a_{2}\bullet b_{3})\cdot S(b_{4})\\&=S(a_{1})\cdot b_{1}\cdot(b_{2}\rightharpoonup T(b_{3})\bullet a_{2}\bullet b_{4})\cdot S(b_{5})\in A^{(n)}
% \end{split}
% \]
since $H^{(n)}$ is a normal Hopf subbrace of $(H,\cdot,\bullet)$ (Corollary \ref{cor:equivalentnotionsnormal}). Therefore, $H^{(n+1)}\subseteq H^{(n)}$. 

\noindent 1). For any $a,b\in H$ and $x\in H^{(n)}$, since $H^{(n)}$ is normal in $H^{\bullet}$, we have $a_{1}\bullet x\bullet T(a_{2})\in H^{(n)}$, hence, by 4) of Lemma \ref{lem:propertiesasterisk}, we have
\[
a\rightharpoonup(x\star b)=(a_{1}\bullet x\bullet T(a_{2}))\star(a_{3}\rightharpoonup b)\in H^{(n)}\star H=H^{(n+1)}.
\]
Since $H^{(n+1)}$ is the Hopf subalgebra of $H^{\cdot}$ generated by elements of the form $x\star b$ with $x\in H^{(n)}$ and $b\in H$,  
%and $a\rightharpoonup p\cdot q=(a_{1}\rightharpoonup p)\cdot(a_{2}\rightharpoonup q)$ for all $a,p,q\in A$, 
we obtain $H\rightharpoonup H^{(n+1)}\subseteq H^{(n+1)}$. 

\noindent 2). For any $a,b\in H$ and $x\in H^{(n)}$, by 1) of Lemma \ref{lem:propertiesasterisk} we have
\[
\begin{split}
a_{1}\cdot(x\star b)\cdot S(a_{2})&=S(x_{1}\star a_{1})\cdot(x_{2}\star a_{2})\cdot a_{3}\cdot(x_{3}\star b)\cdot S(a_{4})\\&=S(x_{1}\star a_{1})\cdot(x_{2}\star (a_{2}\cdot b))\in H^{(n)}\star H=H^{(n+1)}, 
\end{split}
\]
%Since elements $x\star b$ with $x\in A^{(n)}$ and $b\in A$ generate $A^{(n+1)}$ with respect to $\cdot$, it 
hence $H^{(n+1)}$ is normal in $H^{\cdot}$.

\noindent 3). For any $a\in H$ and $y\in H^{(n+1)}$, since $H^{(n+1)}\subseteq H^{(n)}$, we have $y_{1}\cdot(y_{2}\star T(a))\in H^{(n+1)}\cdot (H^{(n)}\star H)\subseteq H^{(n+1)}$. Since $H\rightharpoonup H^{(n+1)}\subseteq H^{(n+1)}$ and $H^{(n+1)}$ is normal in $H^{\cdot}$, using 5) of Lemma \ref{lem:propertiesasterisk} we obtain
\[
a_{1}\bullet y\bullet T(a_{2})=a_{1}\cdot(a_{2}\rightharpoonup(y_{1}\cdot(y_{2}\star T(a_{3}))))\cdot S(a_{4})\in H^{(n+1)},
\]
which proves that $H^{(n+1)}$ is normal in $H^{\bullet}$.
%This proves that $A^{(n)}$ is a normal Hopf subbrace of $A$ for all $n\geq1$.
\end{proof}

Using the operation $\star$, one can also introduce the socle and the annihilator of a (cocommutative) Hopf brace, as it is done for skew braces. The corresponding definitions for skew braces first appeared in \cite[Definition 2.4]{GV} and in \cite[Definition 7]{CCS}, respectively. 
% \begin{remark}
%     We recall that, given a Hopf algebra $A$, the \textit{(left) adjoint action} is the left action of $A$ over itself defined by $a\rhd_{ad}b:=a_{1}bS(a_{2})$, for all $a,b\in A$. Hence, we have that $ab=ba$ if and only if $a_{1}bS(a_{2})=\varepsilon(a)b$, i.e.\ $a\rhd_{ad}b=\varepsilon(a)b$. The center of a group is then replaced, in the Hopf-theoretical setting, by the vector space of $A$ defined by $\mathcal{Z}(A):=\{a\in A\ |\ (-)\rhd_{ad}a=\id_{A}\}$. We observe that this is automatically a subcoalgebra of $A$, in case $A$ is cocommutative. In fact, given $a\in Z$, we have $(a\rhd_{ad}b_{1})\ot a_{2}=a_{1}b_{1}S(a_{2})\ot a_{3}=\varepsilon(a_{1})b\ot a_{2}$ and $a_{1}\ot(a_{2}\rhd_{ad}b)=a_{1}\ot a_{2}bS(a_{3})=a_{3}\ot a_{1}bS(a_{2})=a_{2}\ot\varepsilon(a_{1})b=a_{1}\ot\varepsilon(a_{2})b$. Moreover, $1\in Z$ and, given $b,c\in Z$, we have 
%     \[
% a\rhd_{ad}bc=a_{1}bcS(a_{2})=a_{1}bS(a_{2})a_{3}cS(a_{4})=ba_{1}cS(a_{2})=\varepsilon(a)bc,
%     \]
% so $Z$ is a subalgebra of $A$. Finally, given $b\in Z$, we have that 
% \[
% a\rhd_{ad}S(b)=a_{1}S(b)S(a_{2})=S^2(a_{1})S(b)S(a_{2})=S(a_{2}bS(a_{1}))=S(a_{1}bS(a_{2}))=S(\varepsilon(a)b)=\varepsilon(a)S(b)
% \]
% so it is a Hopf subalgebra of $A$.
% \end{remark}
\begin{remark}
    The notion of center for an arbitrary Hopf algebra $H$ was introduced in \cite{Andruskiewitsch}. If $H$ has bijective antipode (in particular, if $H$ is cocommutative or commutative), the center can be described as
\[
\mathcal{HZ}(H)=\{x\in H\ |\ \Delta(x)\in H\ot\mathcal{Z}(H)\}
\]
where $\mathcal{Z}(H)$ denotes the center of $H$ as an algebra, as proven in \cite[Theorem 2.2]{ChKa}. Observe that %,if $x\in\mathcal{HZ}(A)$, then $x=(\varepsilon\ot\Id)\Delta(x)\in \mathcal{Z}(A)$, i.e.\ 
$\mathcal{HZ}(H)\subseteq\mathcal{Z}(H)$. As proven in \cite[Proposition 4.3.2]{GrKaVe2}, if the base field is algebraically closed of characteristic 0, $\mathcal{HZ}(H)$ coincides with the categorical center of $H$ in the semi-abelian category of cocommutative Hopf algebras over such a field \cite{GrKaVe}. We also observe that, in the cocommutative case, one has $\Delta(x)=\tau\Delta(x)\in\mathcal{Z}(H)\ot H$, where $\tau$ is the ``flip map'', hence $\Delta(x)\in\mathcal{Z}(H)\ot\mathcal{Z}(H)$.
\end{remark}

\begin{remark}\label{rmk:equivalentcondsocle}
We observe that, given a Hopf brace $(H,\cdot,\bullet)$ and a fixed element $a\in H$, we have that $a\star b=\varepsilon(a)\varepsilon(b)1$ for all $b\in H$ if and only if $a\rightharpoonup b=\varepsilon(a)b$ for all $b\in H$. We also notice that, unlike the case of skew braces, the latter conditions are not equivalent--- even when $H$ is cocommutative---to $a\cdot b=a\bullet b$ for all $b\in H$. Indeed, for a fixed $a\in H$, the condition $a\cdot b=a\bullet b$ for all $b\in H$ reads $a_{1}\cdot(a_{2}\rightharpoonup b)=a\cdot b$, and one cannot ``cancel'' $a_{1}$ in this equality to deduce that $a\rightharpoonup b=\varepsilon(a)b$, since $\Delta(a)=a_{1}\ot a_{2}$ is in general a sum of several terms, and the equality $a\cdot b=a\bullet b$ need not hold for the single summands.
\end{remark}

\begin{lemma}\label{lem:HZclosed}
    Let $(H,\cdot,\bullet)$ be a cocommutative Hopf brace. Then, $H\rightharpoonup\mathcal{HZ}(H^{\cdot})\subseteq\mathcal{HZ}(H^{\cdot})$. As a consequence, $\mathcal{HZ}(H^{\cdot})$ is a Hopf strong subbrace of $(H,\cdot,\bullet)$.
\end{lemma}

\begin{proof}
Given $x\in\mathcal{HZ}(H^{\cdot})$ and $a,c\in H$, we have that $\Delta(a\rightharpoonup x)\in H\ot\mathcal{Z}(H^{\cdot})$, since
\[
\begin{split}
    (a_{1}\rightharpoonup x_{1})\ot((a_{2}\rightharpoonup x_{2})\cdot c)&=(a_{1}\rightharpoonup x_{1})\ot\big((a_{2}\rightharpoonup x_{2})\cdot (a_{3}\rightharpoonup(T(a_{4})\rightharpoonup c))\big)\\&=(a_{1}\rightharpoonup x_{1})\ot\big(a_{2}\rightharpoonup x_{2}\cdot(T(a_{3})\rightharpoonup c)\big)\\&=(a_{1}\rightharpoonup x_{1})\ot(a_{2}\rightharpoonup (T(a_{3})\rightharpoonup c)\cdot x_{2})\\&=(a_{1}\rightharpoonup x_{1})\ot\big((a_{2}\rightharpoonup (T(a_{4})\rightharpoonup c))\cdot(a_{3}\rightharpoonup x_{2})\big)\\&=(a_{1}\rightharpoonup x_{1})\ot\big((a_{2}\rightharpoonup (T(a_{3})\rightharpoonup c))\cdot(a_{4}\rightharpoonup x_{2})\big)\\&=(a_{1}\rightharpoonup x_{1})\ot(c\cdot(a_{2}\rightharpoonup x_{2})).
\end{split}
\]
It follows that $a\rightharpoonup x\in \mathcal{HZ}(H^{\cdot})$. Since $\mathcal{HZ}(H^{\cdot})$ is a Hopf subalgebra of $H^{\cdot}$ (see \cite{ChKa}), we get that $\mathcal{HZ}(H^{\cdot})$ is a Hopf strong subbrace of $(H,\cdot,\bullet)$, as desired.
\end{proof}

Given a cocommutative Hopf brace $(H,\cdot,\bullet)$, we define the following vector spaces:
\begin{align}
    \mathrm{soc}(H)&=\{a\in \mathcal{HZ}(H^{\cdot})\ |\ a\star b=\varepsilon(a)\varepsilon(b)1 \text{, for all}\ b\in H\},\label{defsoc}\\
    \mathrm{ann}(H)&=\{a\in \mathcal{HZ}(H^{\cdot})\ |\ a\star b=b\star a=\varepsilon(a)\varepsilon(b)1 \text{, for all}\ b\in H\}.\label{defann}
\end{align}
Clearly, $\mathrm{ann}(H)\subseteq\mathrm{soc}(H)$. We prove the following result:

\begin{lemma}\label{lem:socannsubalg}
    Let $(H,\cdot,\bullet)$ be a cocommutative Hopf brace. Then:
\begin{itemize}
    \item[1)] $\mathrm{soc}(H)$ is a subalgebra of $H^{\bullet}$,
    \item[2)] given $x,y\in\mathrm{ann}(H)$ and $a\in H$, $a\star(x\cdot y)=\varepsilon(a)\varepsilon(xy)1$. 
\end{itemize}    

\end{lemma}

\begin{proof}
%Moreover, since $x_{1}\ot x_{2}\ot x_{3}\in \mathcal{HZ}(A^{\cdot})\ot\mathcal{HZ}(A^{\cdot})\ot\mathcal{HZ}(A^{\cdot})$ and $y_{1}\ot y_{2}\in \mathcal{HZ}(A^{\cdot})\ot\mathcal{HZ}(A^{\cdot})$, we have $\Delta(x\bullet y)=(x_{1}\bullet y_{1})\ot(x_{2}\bullet y_{2})=(x_{1}\bullet y_{1})\ot(x_{2}\cdot(x_{3}\rightharpoonup y_{2}))\in A\ot\mathcal{HZ}(A^{\cdot})$ using Lemma \ref{lem:HZclosed}. Hence $x\bullet y\in \mathcal{HZ}(A^{\cdot})$. 
%Hence, to obtain that $\mathrm{soc}(A)$ is a subalgebra of $A^{\cdot}$ and $A^{\bullet}$, it is enough to prove that $(x\bullet y)\star a=\varepsilon(xy)\varepsilon(a)1$, for all $a\in A$. 
1). Clearly, $1\in\mathrm{soc}(H)$. Given $x,y\in\mathrm{soc}(H)$, we know that $x\bullet y\in\mathcal{HZ}(H^{\cdot})$ by Lemma \ref{lem:HZclosed}. Moreover, since $(y_{1}\star a_{1})\ot(y_{2}\star a_{2})=\Delta(y\star a)=\varepsilon(y)\varepsilon(a)1\ot1$ for all $a\in H$, by 3) of Lemma \ref{lem:propertiesasterisk} we obtain
\[
(x\bullet y)\star a=(x_{1}\star(y_{1}\star a_{1}))\cdot(y_{2}\star a_{2})\cdot(x_{2}\star a_{3})=\varepsilon(y)(x_{1}\star1)\cdot(x_{2}\star a)=\varepsilon(y)x\star a=\varepsilon(xy)\varepsilon(a)1.
\]
2). Given $x,y\in\mathrm{ann}(H)$ and $a\in H$, by 1) of Lemma \ref{lem:propertiesasterisk}, we have 
\[
a\star(x\cdot y)=(a_{1}\star x_{1})\cdot x_{2}\cdot(a_{2}\star y)\cdot S(x_{3})=\varepsilon(y)(a\star x_{1})\cdot x_{2}\cdot S(x_{3})=\varepsilon(y)a\star x=\varepsilon(a)\varepsilon(xy)1. \qedhere
\]
\end{proof}

We define the \textit{socle} of a cocommutative Hopf brace $(H,\cdot,\bullet)$ by
\begin{equation}\label{defSocle}
    \mathrm{Soc}(H)=\{a\in \mathcal{HZ}(H^{\cdot})\ |\ \Delta(a)\in \mathcal{HZ}(H^{\cdot})\ot\mathrm{soc}(H)\} 
\end{equation}
and the \textit{annihilator} of $(H,\cdot,\bullet)$ by
\begin{equation}\label{defAnn}
    \mathrm{Ann}(H)=\{a\in \mathcal{HZ}(H^{\cdot})\ |\ \Delta(a)\in\mathcal{HZ}(H^{\cdot})\ot\mathrm{ann}(H)\}.
\end{equation}

\begin{remark}\label{rmk:propertiesSoc}
    %Given $a\in\mathrm{Soc}(H)$, we have $a=(\varepsilon_{A}\ot\id)\Delta(a)\in\mathrm{soc}(H)$, i.e.\ it satisfies $a\star b=\varepsilon(a)\varepsilon(b)1$ for all $b\in H$. Hence, 
    Clearly, $\mathrm{Soc}(H)\subseteq\mathrm{soc}(H)$. Given $a\in\mathrm{Soc}(H)$, we have 
    \[
    a\bullet b=a_{1}\cdot(a_{2}\rightharpoonup b)=a_{1}\cdot\varepsilon(a_{2})b=a\cdot b, \quad \text{for all } b\in H.
    \]
    Moreover, %we have $\Delta(a)=\tau\Delta(a)\in\mathrm{soc}(H)\ot\mathcal{HZ}(H^{\cdot})$, so 
    by cocommutativity, we have $\Delta(\mathrm{Soc}(H))\subseteq\mathrm{soc}(H)\ot\mathrm{soc}(H)$. Similarly, $\mathrm{Ann}(H)\subseteq\mathrm{ann}(H)$, and $\Delta(\mathrm{Ann}(H))\subseteq\mathrm{ann}(H)\ot\mathrm{ann}(H)$. Moreover, $\mathrm{Ann}(H)\subseteq\mathrm{Soc}(H)$.
\end{remark}

Finally, we prove that $\mathrm{Soc}(H)$ and $\mathrm{Ann}(H)$ are normal Hopf subbraces of $(H,\cdot,\bullet)$.

\begin{proposition}\label{NormalHopf}
Let $(H,\cdot,\bullet)$ be a cocommutative Hopf brace. Then, $\mathrm{Soc}(H)$ and $\mathrm{Ann}(H)$ are normal Hopf subbraces of $(H,\cdot,\bullet)$.
\end{proposition}

\begin{proof}
Clearly, $1\in\mathrm{Ann}(H)\subseteq\mathrm{Soc}(H)$. Given $x,y\in\mathrm{Soc}(H)$, we have $x\cdot y=x\bullet y$, as observed in Remark \ref{rmk:propertiesSoc}. Then, by Lemma \ref{lem:HZclosed} and 1) of Lemma \ref{lem:socannsubalg}, we have $\Delta(x\bullet y)=(x_{1}\bullet y_{1})\ot(x_{2}\bullet y_{2})\in\mathcal{HZ}(H^{\cdot})\ot\mathrm{soc}(H)$. Thus, $\mathrm{Soc}(H)$ is a subalgebra of $H^{\cdot}$ and of $H^{\bullet}$. If $x,y\in\mathrm{Ann}(H)$, we also have $\Delta(x\cdot y)=(x_{1}\cdot y_{1})\ot(x_{2}\cdot y_{2})\in\mathcal{HZ}(H^{\cdot})\ot\mathrm{ann}(H)$ by 2) of Lemma \ref{lem:socannsubalg}. It then follows that also $\mathrm{Ann}(H)$ is a subalgebra of $H^{\cdot}$ and $H^{\bullet}$.
%Given $a\in\mathrm{Soc}(A)$, we know that $\Delta(a)\in\mathcal{HZ}(A^{\cdot})\ot\mathcal{HZ}(A^{\cdot})$ since $\mathcal{HZ}(A^{\cdot})$ is a subcoalgebra of $A^{\cdot}$. 
Moreover, given $x\in\mathrm{Soc}(H)$, we have
\[
\begin{split}
    (\Id\ot\Delta)\Delta(x)=(\Delta\ot\Id)\Delta(x)\in\mathcal{HZ}(H^{\cdot})\ot\mathcal{HZ}(H^{\cdot})\ot\mathrm{soc}(H),
\end{split}
\]
hence $\Delta(x)\in\mathrm{soc}(H)\ot\mathrm{Soc}(H)$. By cocommutativity, we get $\Delta(x)\in\mathrm{Soc}(H)\ot\mathrm{Soc}(H)$, i.e.\ $\mathrm{Soc}(H)$ is a subcoalgebra of $H$. Analogously, $\mathrm{Ann}(H)$ is a subcoalgebra of $H$. Given $x\in\mathrm{Soc}(H)$ and $a\in H$, by 5) of Lemma \ref{lem:propertiesasterisk} and Lemma \ref{lem:HZclosed}, we have 
\[
\begin{split}
a_{1}\bullet x\bullet T(a_{2})&=a_{1}\cdot(a_{2}\rightharpoonup(x_{1}\cdot(x_{2}\star T(a_{3}))))\cdot S(a_{4})=a_{1}\cdot(a_{2}\rightharpoonup(x_{1}\cdot\varepsilon(x_{2})\varepsilon T(a_{3})1))\cdot S(a_{4})\\&=a_{1}\cdot(a_{2}\rightharpoonup x)\cdot S(a_{3})=a_{1}\cdot S(a_{3})\cdot(a_{2}\rightharpoonup x)=a_{1}\cdot S(a_{2})\cdot(a_{3}\rightharpoonup x)=a\rightharpoonup x,
\end{split}
\]
hence, given $b\in H$, we get
\[
\begin{split}
    (a\rightharpoonup x)\rightharpoonup b=(a_{1}\bullet x\bullet T(a_{2}))\rightharpoonup b&=a_{1}\rightharpoonup(x\rightharpoonup(T(a_{2})\rightharpoonup b))=\varepsilon(x)a_{1}\rightharpoonup(T(a_{2})\rightharpoonup b)%\\&=\varepsilon(x)a_{1}\bullet T(a_{2})\rightharpoonup b
=\varepsilon(x)\varepsilon(a)b,
\end{split}
\]
i.e.\ $a\rightharpoonup x=a_{1}\bullet x\bullet T(a_{2})\in\mathrm{soc}(H)$, for any $x\in\mathrm{Soc}(H)$ and $a\in H$. Hence, since $\mathrm{Soc}(H)$ is a subcoalgebra of $H$, we have %$(a_{1}\rightharpoonup x_{1})\ot(a_{2}\rightharpoonup x_{2})=(a_{1}\rightharpoonup x_{1})\ot(a_{2}\bullet x_{2}\bullet T(a_{3}))$, and we get 
$(a_{1}\rightharpoonup x_{1})\ot((a_{2}\rightharpoonup x_{2})\rightharpoonup b)%=(a_{1}\rightharpoonup x_{1})\ot(a_{2}\bullet x_{2}\bullet T(a_{3})\rightharpoonup b)
=(a_{1}\rightharpoonup x_{1})\ot\varepsilon(a_{2}\rightharpoonup x_{2})b$, for any $b\in H$, i.e.\ $a\rightharpoonup x\in\mathrm{Soc}(H)$. If $x\in\mathrm{Ann}(H)$ and $a\in H$, since $\mathrm{Ann}(H)$ is a subcoalgebra of $H$, we get $a\rightharpoonup x=(a\star x_{1})\cdot x_{2}=\varepsilon(a)x$ and then $\Delta(a\rightharpoonup x)=\varepsilon(a)x_{1}\ot x_{2}\in\mathrm{Ann}(H)\ot\mathrm{Ann}(H)$. Given $x\in\mathrm{Soc}(H)$ and $b\in H$, we get
\begin{align*}
T(x)&=S(x_{1})\cdot x_{2}\cdot T(x_{3})=S(x_{1})\cdot(x_{2}\bullet T(x_{3}))=x_{1}\rightharpoonup T(x_{2})=S(x),\\
T(x)\rightharpoonup b&=T(x_{1})\rightharpoonup\varepsilon(x_{2})b=T(x_{1})\rightharpoonup(x_{2}\rightharpoonup b)=T(x_{1})\bullet x_{2}\rightharpoonup b=\varepsilon(x)b,
\end{align*}
hence $\mathrm{Soc}(H)$ is a Hopf strong subbrace of $(H,\cdot,\bullet)$. The same can be deduced for $\mathrm{Ann}(H)$. We have already proven that $\mathrm{Soc}(H)$ and $\mathrm{Ann}(H)$ are normal Hopf subalgebras of $H^{\bullet}$ and, since $\mathrm{Ann}(H)\subseteq\mathrm{Soc}(H)\subseteq\mathcal{Z}(H^{\cdot})$, they are also normal Hopf subalgebras of $H^{\cdot}$.
\end{proof}

We leave for future work the study of socle and annihilator series for cocommutative Hopf braces, as it is done in \cite{CSV} and \cite{BP}  for skew braces, respectively.

\begin{section}{The Birkhoff subcategory of cocommutative Hopf algebras}\label{sec: relative commutator}

We now want to study central extensions (in the sense of \cite{JK}) with respect to the adjunction \eqref{Adj HBr Hopf}. The left adjoint $F\colon \HBrc\to\Hopfc$ has been described in \cite[Remark 8.9]{GranSciandra} where $\Hopfc$ is proven to be a \textit{Birkhoff subcategory} of $\HBrc$, namely a full (replete) reflective subcategory that is closed under subobjects and regular quotients.
% \begin{equation}
% \begin{tikzcd}
% \HBrc &&& \Hopfc
% 	\arrow[" "{name=1, anchor=center, inner sep=0}, curve={height=-8pt}, from=1-1, to=1-4,"F"]
% 	\arrow[" "{name=1, anchor=center, inner sep=0}, curve={height=-8pt}, from=1-4, to=1-1, "G"]
%     \arrow["\vdash"{anchor=center, rotate=90}, draw=none, from=1-4, to=1-1].
% \end{tikzcd}
% \end{equation}
% \begin{remark}\label{Birkhoff}
%     It was shown in \cite{GranSciandra} that 
% \end{remark}
% \begin{figure}[H]
%     \centering
%     \begin{tikzpicture}
%         \node(0) at (-1.5,0) {$\HBrc$};
%         \node(1) at (1.5,0) {$\Hopfc$};
%         \node(2) at (0,0) {$\perp$};
%         \draw[->] (0) to[out=20,in=160] node[above]{$F$} (1);
%         \draw[->] (1) to[out=200,in=-20] node[below]{$G$} (0);
        
%     \end{tikzpicture}
%     \caption{Adjunction between cocommutative Hopf braces and cocommutative Hopf algebras}
% \end{figure}

\begin{definition}\label{def: commutator hopfc}
    Let $(H,\cdot,\bullet)$ be a cocommutative Hopf brace and $(I,\cdot,\bullet)$ a normal Hopf subbrace of $(H,\cdot,\bullet)$. We define $[I,H]_{\Hopfc}$ as the Hopf subalgebra of $H^{\cdot}$ generated by 
    $$\{i\star h,\; h\star i,\; k_1\cdot (h\star i)\cdot S(k_2)\mid i\in I,\; \; h,k\in H\}.$$
\end{definition}

\begin{proposition}\label{prop:relativecommnormal}
    %For any $I\subseteq H$ normal Hopf subbrace, 
    The Hopf subalgebra $[I,H]_\Hopfc$ of $H^{\cdot}$ is a normal Hopf subbrace of $(H,\cdot,\bullet)$.
\end{proposition}
\begin{proof}
Since $I$ is a normal Hopf subbrace, by Corollary \ref{cor:equivalentnotionsnormal} we know that $[I,H]_\Hopfc\subseteq I$. Then, for any $h\in H$ and $b\in [I,H]_\Hopfc$, %$(h\rightharpoonup b_1)\cdot S(b_2)=h\star b\in [I,H]_\Hopfc$, which implies 
$h\rightharpoonup b=%(h\rightharpoonup b_1)\cdot S(b_2)\cdot b_3=
(h\star b_{1})\cdot b_{2}\in [I,H]_\Hopfc$.
    Hence, %$H\rightharpoonup[I,H]_\Hopfc\subseteq [I,H]_\Hopfc$ and so
 by Remark \ref{rmk:Hopfsubbullet}, we have that $[I,H]_\Hopfc$ is a Hopf subbrace of $(H,\cdot,\bullet)$. It remains to prove that $h_1\cdot b\cdot S(h_2)\in [I,H]_\Hopfc$ and $h_1\bullet b\bullet T(h_2)\in [I,H]_\Hopfc$ for any $h\in H$ and $b\in [I,H]_\Hopfc$. 
    % The first statement immediately follows from the definition of $[I,H]_\Hopfc$ and from the fact that 
    % $$l_1\cdot k_1\cdot (h\star i)\cdot S(k_2)\cdot S(l_2)=(l\cdot k)_1\cdot (h\star i)\cdot S((l\cdot k)_2)$$
    % for any $l,k,h\in H$ and $i\in I$. 
    The elements $h\star i$, with $h\in H$ and $i\in I$, are closed under the adjoint action of $H^{\cdot}$ by definition of $[I,H]_\Hopfc$. Moreover, by 1) of Lemma \ref{lem:propertiesasterisk}, %we have that $ (i_{1}\star x_{1})\cdot x_{2}\cdot(i_{2}\star y)\cdot S(x_{3})=i\star(x\cdot y)\in[I,H]_\Hopfc$, for any $i\in I$ and $x,y\in H$. Therefore, 
    we have
\[
x_{1}\cdot(i\star h)\cdot S(x_{2})=S(i_{1}\star x_{1})\cdot(i_{2}\star x_{2})\cdot x_{3}\cdot(i_{3}\star h)\cdot S(x_{4})=S(i_{1}\star x_{1})\cdot(i_{2}\star(x_{2}\cdot h))\in[I,H]_\Hopfc,
\]
%for any $x,y\in H$, $i\in I$.
%and hence $[I,H]_\Hopfc$ is a normal Hopf subalgebra of $H^{\cdot}$.
for any $x\in H$. Finally, by $5)$ of Lemma \ref{lem:propertiesasterisk}, we have
\[
    h_1\bullet b\bullet T(h_2)= h_1\cdot (h_2\rightharpoonup(b_1\cdot(b_2\star T(h_3))))\cdot S(h_4)\in [I,H]_\Hopfc,
\]
for any $h\in H$ and $b\in[I,H]_\Hopfc$, since $H\rightharpoonup[I,H]_\Hopfc\subseteq [I,H]_\Hopfc$ and $[I,H]_\Hopfc$ is a normal Hopf subalgebra of $H^{\cdot}$. 
\end{proof}

\begin{proposition}\label{centralExt}
    A regular epimorphism $f\colon (A,\cdot,\bullet)\to(B,\cdot,\bullet)$ in $\HBrc$ is a central extension with respect to the adjunction \eqref{Adj HBr Hopf} if and only if $[\Hker (f),A]_\Hopfc=\Bbbk1_{A}$, where $\mathrm{Hker}(f)=\{x\in A\ |\ x_{1}\ot f(x_{2})=x\ot1_{B}\}$ is the kernel of $f$ in $\HBrc$. 
\end{proposition}
\begin{proof}
Let us denote by $(\Eq (f),s,t)$ the kernel pair of $f$ in $\HBrc$ and by $R(X)$ the kernel in $\HBrc$ of the $X$-component $\eta_X\colon X\to GF(X)$ of the unit $\eta$ of the adjunction \eqref{Adj HBr Hopf}, namely
    \[
    \Eq (f)=\{a\otimes c\in A\otimes A\mid a_1\otimes f(a_2)\otimes c=a\otimes f(c_1)\otimes c_2\},
    \]
    see e.g.\ \cite[page 275]{GranSciandra}, and $R(X)=X\star X$ for any $(X,\cdot,\bullet)$ in $\HBrc$. The assignment $R$ defines a functor and the extension $f$ is central if and only if $R(s)=R(t)$ (see \cite{BournGran}).
% \[\begin{tikzcd}
% 	R(\mathrm{Eq}(f)) & R(A) & \\
% 	\mathrm{Eq}(f) & A & B
% 	\arrow[shift left, from=1-1, to=1-2,"R(s)"]
% 	\arrow[shift right, from=1-1, to=1-2,"R(t)"']
% 	\arrow[tail, from=1-1, to=2-1]
% 	\arrow[tail, from=1-2, to=2-2]
% 	\arrow[shift left, from=2-1, to=2-2,"s"]
% 	\arrow[shift right, from=2-1, to=2-2,"t"']
% 	\arrow[from=2-2, to=2-3,"f"]
% \end{tikzcd}\]
Firstly, we assume that the regular epimorphism $f$ is a central extension and we prove that $[\Hker (f), A]_\Hopfc$ is trivial. Let $h\in \Hker (f)$ and $a\in A$. Since $h\in \Hker (f)$, we have $h_1\otimes f(h_2)\otimes 1=h\otimes 1\otimes 1=h\otimes f(1)\otimes 1$,
    therefore $h\otimes 1\in\Eq(f)$.
    Furthermore, $a_1\otimes a_2\in \Eq (f)$ as well, hence we have
    \begin{align*}
       R(\Eq (f))\ni(h\otimes 1)\star_{A\ot A} (a_1\otimes a_2)=& (S(h_1)\cdot (h_2\bullet a_1)\cdot S(a_2))\otimes (1\cdot (1\bullet a_3)\cdot S(a_4))=(h\star a)\otimes 1.
    \end{align*}
    Using the assumption that $R(s)=R(t)$, we obtain
    $$h\star a= R(s)((h\star a)\otimes 1)=R(t)((h\star a)\otimes 1)=\varepsilon(h)\varepsilon(a)1.$$
    In the same way, one shows that $a\star h=\varepsilon(h)\varepsilon(a)1$, hence $[\Hker (f),A]_\Hopfc$ is trivial.

    Conversely, we now assume that $[\Hker (f), A]_\Hopfc=\Bbbk1$, and we show that $R(s)=R(t)$. First of all, %we notice that for any $h\in \Hker (f)$ and any $x\in A$, $[\Hker (f), A]_\Hopfc=\Bbbk1$ implies that $S(x_1)\cdot(x_2\bullet h_1)\cdot S(h_2)= x \star h=\varepsilon(x)\varepsilon(h)1$. As a consequence, 
    since $x \star h=\varepsilon(x)\varepsilon(h)1$ for all $x\in A$ and $h\in \Hker (f)$, we have \begin{equation}\label{eq cdot=bullet}
        x\cdot h=x_1\cdot\varepsilon(x_2)\varepsilon(h_1)1\cdot h_2=x_{1}\cdot(x_{2}\star h_{1})\cdot h_{2}=x_1\cdot S(x_2)\cdot(x_3\bullet h_1)\cdot S(h_2)\cdot h_3=x\bullet h.
    \end{equation}
Similarly, $h\cdot x=h\bullet x$ for all $x\in A$ and $h\in \Hker(f)$.
Now, given $a\otimes b,c\otimes d\in \Eq (f)$, we need to prove $R(s)((a\otimes b)\star_{A\ot A} (c\otimes d))=R(t)((a\otimes b)\star_{A\ot A} (c\otimes d))$, i.e.\ $R(s)((a\star c)\ot (b\star d))=R(t)((a\star c)\ot (b\star d))$. Since $(\mathrm{Id}\ot\Delta)(\mathrm{Eq}(f))\subseteq\mathrm{Eq}(f)\ot A$, it suffices to show that $(a\star c)\cdot S(b\star d)=\varepsilon(a)\varepsilon(b)\varepsilon(c)\varepsilon(d)1$, for $a\otimes b  \in \Eq (f),c\otimes d\in \Eq (f)$. %where equality (!) relies on the fact that $a\otimes b\in \Eq (f)$, thus $a_1\otimes f(a_2)\otimes b=a\otimes f(b_1)\otimes b_2$ which implies 
Since %$a_1\otimes a_2\otimes f(a_3)\otimes b=a_1\otimes a_2\otimes f(b_1)\otimes b_2$, 
$a_{1}\ot f(a_{2})\ot b_{1}\ot b_{2}\overset{(!)}{=}a\ot f(b_{1})\ot b_{2}\ot b_{3}$, we get 
\begin{align*}
        (S(a_1)\cdot b_1)\otimes f(S(a_2)\cdot b_2)&= 
        (S(a_1)\cdot b_1)\otimes(Sf(a_2)\cdot f(b_2))
        \overset{(!)}{=} (S(a)\cdot b_2)\otimes (Sf(b_1)\cdot f(b_{3}))\\&=(S(a)\cdot b)\ot1,
    \end{align*}   
hence
\begin{equation}\label{eq Sa.b in Hker}
        S(a)\cdot b\in \Hker (f), \qquad \text{for any } a\ot b\in\mathrm{Eq}(f).
    \end{equation}
Similarly, $a\cdot S(b)\in\mathrm{Hker}(f)$, for any $a\ot b\in\mathrm{Eq}(f)$. Since $(\Delta\ot\mathrm{Id})(\mathrm{Eq}(f))\subseteq A\ot\mathrm{Eq}(f)$, as a consequence of \eqref{eq cdot=bullet} and \eqref{eq Sa.b in Hker} we obtain 
\begin{equation}\label{eq Sa.b=Ta bullet b}
S(a)\cdot b=T(a_{1})\bullet a_{2}\bullet (S(a_3)\cdot b)=T(a_{1})\bullet( a_2\cdot S(a_3)\cdot b)=T(a)\bullet b,\quad \text{for any } a\ot b\in\mathrm{Eq}(f).
\end{equation}
Similarly, $a\cdot S(b)=a\bullet T(b)$, for any $a\ot b\in\mathrm{Eq}(f)$. Given $a\otimes b,c\otimes d\in \Eq (f)$, we can now compute the following, where we also use the fact that $\Eq(f)$ is a subcoalgebra of $A\ot A$, so that $a_{1}\ot b_{1}$ and $a_{2}\ot b_{2}$ (and similarly for $c\ot d$) are elements of $\Eq(f)$:
    \begin{align*}
        (a\star c)\cdot S(b\star d)&= S(a_1)\cdot(a_2\bullet c_1)\cdot S(c_2)\cdot S(S(b_1)\cdot(b_2\bullet d_1)\cdot S(d_2))\\
        &=S(a_1)\cdot(a_2\bullet c_1)\cdot S(c_2)\cdot d_2\cdot S(b_2\bullet d_1)\cdot b_1\\&\overset{\eqref{eq cdot=bullet},\eqref{eq Sa.b in Hker}}{=}S(a_1)\cdot((a_2\bullet c_1)\bullet( S(c_2)\cdot d_2))\cdot S(b_2\bullet d_1)\cdot b_1\\
        &\overset{\eqref{eq Sa.b=Ta bullet b}}{=}S(a_1)\cdot (a_2\bullet c_1\bullet T(c_2)\bullet d_2)\cdot S(b_2\bullet d_1)\cdot b_1\\
        &=\varepsilon(c)S(a_1)\cdot (a_2\bullet d_2)\cdot S(b_2\bullet d_1)\cdot b_1\\
        &=\varepsilon(c)S(a_1)\cdot (a_2\bullet T(b_1)\bullet b_2\bullet d_2)\cdot S(b_3\bullet d_1)\cdot b_4\\
        &%\overset{\eqref{eq cdot=bullet}}{=}
=\varepsilon(c)S(a_1)\cdot (a_2\bullet T(b_1))\cdot( b_2\bullet d_2)\cdot S(b_3\bullet d_1)\cdot b_4\\
        &=\varepsilon(c)\varepsilon(d) S(a_1)\cdot (a_2\bullet T(b_1))\cdot b_2\\&=
        %\overset{\eqref{eq cdot=bullet}}{=}
\varepsilon(c)\varepsilon(d)S(a_1)\cdot (a_2\bullet T(b_1)\bullet b_2)\\
        &=\varepsilon(b)\varepsilon(c)\varepsilon(d) S(a_1)\cdot a_2=\varepsilon(a)\varepsilon(b)\varepsilon(c)\varepsilon(d)1.
    \end{align*}
    This concludes the proof.
\end{proof}

\begin{remark}\label{Izquierdo}
    In \cite{Izquierdo}, Pérez-Izquierdo investigated the process of passing from terms in a variety of universal algebras to what he called ``linearized terms'', that are homomorphisms in the category $\mathsf{Coalg}_{\mathsf{coc}}$ of cocommutative coalgebras. Under this ``linearizing process'', equalities in the category of skew braces translate into equalities in the category of %cocommutative coalgebras/
    cocommutative Hopf braces. Some of the computations in this article involve these linearized terms, and the resulting identities can be interpreted as the linearization of the corresponding equations in the category of skew braces. For instance, the equations in Lemma \ref{lem:propertiesasterisk} are the linearized version of the identities presented in \cite[Lemma 2.7]{Tsang}. Therefore, the validity of Lemma \ref{lem:propertiesasterisk}---which we proved in a direct way for the sake of completeness---can be also deduced from \cite[Theorem 13]{Izquierdo}. Indeed, in this theorem the author shows that, whenever an equation holds in a variety of universal algebras of type $\mathcal{F}$, then its ``linearized version'' holds for $\mathcal{F}$-coalgebras, namely cocommutative coalgebras equipped with an $\mathcal{F}$-algebra structure which is compatible with the coalgebra structure. 

    Nevertheless, not all computations can be straightforwardly deduced in this way, but just those identities that are equations involving arbitrary elements of a cocommutative Hopf brace. The last computation of Proposition \ref{centralExt}, for instance, holds just for suitable $a,b,c, d\in A$: it is not a linearized equation in the sense of Pérez-Izquierdo, because the corresponding ``sentence'' in the category of skew braces should involve an ``implication''.
\end{remark}

An application of Pérez-Izquierdo approach is that any class of $\mathcal F$-coalgebras described by linearized equations can be interpreted as a category of coalgebraic models for a suitable algebraic theory (see \cite{Bev} for more details). We apply this observation in the following. 
For any fixed $n\in {\mathbb N}^*$, we denote by $\mathsf{Nil}^{(n)} $ the full subcategory of $\HBrc$ whose objects are the cocommutative Hopf braces $(A,\cdot,\bullet)$ such that $[A^{(n)},A]_{\Hopfc}= \Bbbk1_{A} $; as observed in the Introduction, such Hopf braces are in particular right nilpotent of class at most $n$, in the sense that $A^{(n+1)}=\Bbbk1_{A}$.
 We are now going to show that $\mathsf{Nil}^{(n)} $ is a Birkhoff subcategory of the semi-abelian category $\HBrc$ (one could also deduce this result from \cite[Proposition $7.8$]{TT} together with Proposition \ref{centralExt}):
\begin{proposition}\label{Birkhoff}
 For any $n\in {\mathbb N}^*$,  $\mathsf{Nil}^{(n)} $ is a Birkhoff subcategory of $\HBrc$.
\end{proposition}
\begin{proof}
Let us denote by $+,\circ,-,()^{-1}$ the operations in the theory of skew braces and let $(A,\cdot,\bullet)$ be a cocommutative Hopf brace. We notice that the homomorphisms of coalgebras $\cdot$, $\bullet\colon A\otimes A\to A$ and $S,T\colon A\to A$ are linearized terms in the sense of \cite{Izquierdo}, since they are the images of the terms ``$x+y$'', ``$x\circ y$'', ``$-x$'' and ``$x^{-1}$'' respectively, where $x$ and $y$ denote two variables. Furthermore, the star operation $\star\colon A\otimes A\to A$ is also a linearized term because it is the image of term ``$-x+(x\circ y)-y$'' through the linearizing map from the term algebra to $\mathrm{Hom}(A\otimes A,A)$.
We recall that $A^{(n)}$ is generated by elements of the form $(\cdots((a^1\star a^2)\star a^3)\star\cdots )\star a^n$ for $a^1,\dots, a^n\in A$, therefore we explicitly know a set of generators for $[A^{(n)},A]_\Hopfc$ (see Definition \ref{def: commutator hopfc}). These generators are linearized terms as well because they are all obtained by composing the linearized terms $\cdot,\bullet,S,T$ and $\star$. Thus the subcategory $\mathsf{Nil}^{(n)}$ can be described using the linearized equations stating that these terms are trivial.
This means that the subcategory $\mathsf{Nil}^{(n)}$ can be interpreted as a category of coalgebraic models for a suitable algebraic theory, and the inclusion functor $\mathsf{Nil}^{(n)}\to \HBrc$ is an algebraic functor in the sense of Lawvere \cite{Lawvere}, which is induced by a surjective map of theories. For this reason, the thesis immediately  follows from \cite[Proposition 6.3]{Bev}.
\end{proof}

\section{Weak $\mathcal E$-universal central extensions and homology}\label{section: Weak univ centr ext}

We are now going to prove the existence of a \emph{weak $\mathcal E$-universal central extension} of cocommutative Hopf braces following the approach based on \emph{cleft extensions} developed in \cite{GranSciandra-2} for cocommutative Hopf algebras.
\begin{definition}
A surjective morphism $f \colon (A, \cdot, \bullet) \rightarrow (B, \cdot, \bullet)$ of cocommutative Hopf braces is called a \emph{cleft extension} if there exists a morphism $s \colon B \rightarrow A$ of cocommutative coalgebras such that $fs = \mathrm{Id}_B$.
\end{definition}
 The class $\mathcal E$ of cleft extensions of cocommutative Hopf braces satisfies the following properties, that are proved exactly as in the case of cocommutative Hopf algebras (see \cite[Lemma 3.6]{GranSciandra-2}):
\begin{itemize}
    \item[i)] $\mathcal{E}$ contains the isomorphisms in $\HBrc$,
    \item[ii)] $\mathcal{E}$ is pullback stable, 
    \item[iii)] $\mathcal{E}$ is closed under composition,
    \item[iv)] if $g f$ is in $\mathcal{E}$, then $g$ is also in $\mathcal{E}$. 
\end{itemize} 
Thanks to the fact that the category $\HBrc$ is monadic on the category $\mathsf{Coalg}_{\mathsf{coc}}$ \cite[Theorem 3.1]{Agore} we know that, for any cocommutative Hopf brace 
$(B, \cdot, \bullet)$, there exists a regular epimorphism $\epsilon_B  \colon (TV(B), \cdot, \bullet) \rightarrow (B, \cdot, \bullet)$, where $V$ and $T$ are the forgetful and the free functor, respectively, in the adjunction
\begin{equation}\label{adj}
\begin{tikzcd}
\mathsf{Coalg}_{\mathsf{coc}} &&& \HBrc
	\arrow[" "{name=1, anchor=center, inner sep=0}, curve={height=-8pt}, from=1-1, to=1-4,"T"]
	\arrow[" "{name=1, anchor=center, inner sep=0}, curve={height=-8pt}, from=1-4, to=1-1, "V"]
    \arrow["\vdash"{anchor=center, rotate=90}, draw=none, from=1-4, to=1-1].
\end{tikzcd}
\end{equation}
Indeed, the triangular identity
$V(\epsilon_B)\eta_{V( B)} =\Id_{V(B)}$, where $\eta_{V(B)}$ is the $V(B)$-component of the unit $\eta$ of the adjunction \eqref{adj}, shows that the $B$-component $\epsilon_B$ of the counit $\epsilon$ is split as a morphism of cocommutative coalgebras. This implies that $\epsilon_B %\colon TV(B) \rightarrow B
$ is a regular epimorphism in $\HBrc$ (since it is surjective, see \cite[Corollary 4.8]{GranSciandra}), and it belongs to the class $\mathcal E$ of cleft extensions.
The cleft extension $\epsilon_B  \colon (TV(B), \cdot, \bullet) \rightarrow (B, \cdot, \bullet)$ is $\mathcal E$-projective, in the following sense: given any other morphism $g \colon (E,\cdot,\bullet) \rightarrow (B,\cdot,\bullet)$ in $\mathcal E$, there exists a morphism $\phi \colon (TV(B),\cdot,\bullet) \rightarrow (E,\cdot,\bullet)$ in $\HBrc$ such that $g \phi = \epsilon_B$.
%\[\begin{tikzcd}
%	& E \\
%	TV(B) & B.
%	\arrow[from=1-2, to=2-2, "g"]
%	\arrow[dashed, from=2-1, %to=1-2, "\phi"]
%	\arrow[from=2-1, to=2-2, %"\varepsilon_B"']
% \end{tikzcd}\]
The $\mathcal E$-projectivity of $TV(B)$ is a consequence of the following result, whose proof is similar to the one of \cite[Lemma 3.5]{GranSciandra-2}.

\begin{lemma}\label{lem:projectiveobjects}
For any $C$ in $\mathsf{Coalg}_{\mathsf{coc}}$, the free Hopf brace $T(C)$ is $\mathcal{E}$-projective.
    %i.e. for any morphism $\phi:F(C)\to N$ in $\mathsf{Hopf_{coc}}$ and any morphism $p:M\to N$ in $\mathcal{E}$, there exists a morphism $\psi:F(C)\to M$ in $\mathsf{Hopf_{coc}}$ such that $p\psi=\phi$:
%\[\begin{tikzcd}
	%& M \\
	%F(C) & N
	%\arrow[from=1-2, to=2-2, "p"]
	%\arrow[dashed, from=2-1, to=1-2, "\psi"]
	%\arrow[from=2-1, to=2-2, "\phi"']
%\end{tikzcd}\]
\end{lemma}

\begin{proof}
Let $C$ be an object in $\mathsf{Coalg}_{\mathrm{coc}}$, $\psi:T(C)\to B$ a morphism in $\HBrc$ and $g:E\to B$ a morphism in $\mathcal{E}$ such that $V(g)$ has a section $i:V(B)\to V(E)$ in $\mathsf{Coalg}_{\mathrm{coc}}$. Then, there is a morphism $i V(\psi) \eta_{C}:C\to V(E)$ in $\mathsf{Coalg}_{\mathrm{coc}}$, where $\eta_{C}$ is the $C$-component of the unit $\eta$ of the adjunction \eqref{adj}. There is then a unique morphism $\tau \colon T(C)\to E$ in $\HBrc$ such that $V(\tau) \eta_C =i V(\psi) \eta_C$, and then it follows $g \tau = \psi$, since $V(g)  i = 1_{V(B)}$.
\end{proof}

We are now able to prove the following result:

\begin{proposition}\label{thm: weak E-univ centr extension}
    Given a cocommutative Hopf brace $(B, \cdot, \bullet)$ there exists a weak $\mathcal E$-universal central extension $f \colon (A,\cdot,\bullet) \rightarrow (B,\cdot,\bullet)$ of $(B,\cdot,\bullet)$.
\end{proposition}
\begin{proof}
Given any cocommutative Hopf brace $(B, \cdot, \bullet)$, we consider its  ``free presentation'' given by 
$ \epsilon_B \colon (TV(B),\cdot,\bullet) \rightarrow (B,\cdot,\bullet)$. 
One then defines the quotient 
\[
\pi:\begin{tikzcd}
TV(B) &  \frac{TV(B)}{TV(B)\cdot[ \mathrm{Hker}(\epsilon_B), TV(B) ]_{\Hopfc}^+}=:A 
\arrow[from=1-1, to=1-2],
\end{tikzcd}
\]
where $[ \mathrm{Hker}(\epsilon_B), TV(B) ]_{\mathsf{Hopf}_{\mathsf{coc}}}$
is the relative commutator defined as in Definition \ref{def: commutator hopfc}, which is a normal Hopf subbrace (Proposition \ref{prop:relativecommnormal}), so that $\pi$ is in $\HBrc$. Since $\pi$ is the cokernel of the inclusion ${[ \mathrm{Hker}(\epsilon_B), TV(B) ]_{\Hopfc}}\hookrightarrow TV(B)$ in $\HBrc$, there is a unique morphism $f$ in $\HBrc$ such that $f\pi = \epsilon_B$.
\begin{figure}[H]
    \centering
    \begin{tikzpicture}
        \node(0) at (-1,0) {$[\Hker(\epsilon_B),TV(B)]_\Hopfc$};
        \node(1) at (3,0) {$TV(B)$};
        \node(2) at (5,0) {$A$};
        \node(3) at (4.5,-1) {$B$};

        \draw[right hook ->] (0)--(1);
        \draw[->] (1) to node[above]{\scriptsize{$\pi$}} (2);
        \draw[->] (1) to node[left]{\scriptsize{$\epsilon_B$}} (3);
        \draw[dashed, ->] (2) to node[right]{\scriptsize{$f$}} (3);   
    \end{tikzpicture}
\end{figure}
\noindent The morphism $\epsilon_B$ is in $\mathcal E$, hence $f$ is in $\mathcal E$ by the property iv)  of the class $\mathcal E$ of cleft extensions recalled above.
To explain why this extension $f$ is central in the sense of Proposition \ref{centralExt}, consider the following commutative diagram of short exact sequences in $\HBrc$
\begin{equation}\label{SES}  
\begin{tikzcd}
	0 & (\mathrm{Hker}(\epsilon_{B}),\cdot,\bullet) & (TV(B),\cdot,\bullet) & (B,\cdot,\bullet) & 0\\
0 & (\mathrm{Hker}(f),\cdot,\bullet) & (A,\cdot,\bullet) & (B,\cdot,\bullet) & 0,
	\arrow[hook, from=1-2, to=1-3, "\mathsf{ker}(\epsilon_{B})"]
	\arrow[from=1-2, to=2-2, "\overline{\pi}"']
%\arrow["\lrcorner"{anchor=center, pos=0.125}, draw=none, from=1-1, to=2-2]
	\arrow[from=1-3, to=2-3, "\pi"]
	\arrow[from=1-3, to=1-4, "\epsilon_{B}"]
    \arrow[from=1-4, to=1-5, "{}"]
    \arrow[from=1-1, to=1-2, "{}"]
    \arrow[from=1-4, to=2-4, "{1_B}"]
    \arrow[from=2-4, to=2-5, "{}"]
    \arrow[from=2-1, to=2-2, "{}"]
    \arrow[from=2-3, to=2-4, "f"']
    \arrow[hook, from=2-2, to=2-3, "\mathsf{ker}(f)"']
\end{tikzcd}
\end{equation}
where $\overline{\pi}$ is the restriction of the morphism $\pi$ to $\mathrm{Hker}(\epsilon_{B})$. The left-hand square in this diagram is a pullback, since  the horizontal sequences are exact and the right-hand vertical morphism $1_B$ is a monomorphism (see \cite[Lemma 1]{Bourn}).
It follows that $\overline{\pi}:(\mathrm{Hker}(\epsilon_{B}),\cdot,\bullet)\to(\mathrm{Hker}(f),\cdot,\bullet)$ is a regular epimorphism in $\HBrc$ (i.e.\ it is surjective%by \cite[Corollary 4.8]{GranSciandra}
), since it is the pullback of  the regular epimorphism $\pi$ along $\mathsf{ker}(f)$ in the semi-abelian category $\HBrc$. It follows that $\pi(\mathrm{Hker}(\epsilon_{B}))=\mathrm{Hker}(f),$ and
\[
[\mathrm{Hker} (f), A]_{\Hopfc} = [ \pi (\mathrm{Hker} (\epsilon_B)), \pi (TV(B))]_{\Hopfc} = \pi( [\mathrm{Hker} (\epsilon_B), TV(B)]_{\Hopfc})=\Bbbk1_{A},
\]
\normalsize
where the second equality is a direct consequence of \cite[Proposition 3.2]{EG} applied to the right-hand square  in diagram \eqref{SES}.
%which is clearly a pushout of regular epimorphisms.
Now that we have proved that the extension $f \colon (A,\cdot,\bullet) \rightarrow (B,\cdot,\bullet)$ is central in the sense of Proposition \ref{centralExt}, it remains to show that it is weak $\mathcal E$-universal. For this, let $g \colon (C,\cdot,\bullet) \rightarrow (B,\cdot,\bullet)$ be another central extension of $(B,\cdot,\bullet)$ in $\mathcal E$. The fact that $TV(B)$ is $\mathcal E$-projective (Lemma \ref{lem:projectiveobjects}) %and $g \in \mathcal E$ 
implies that there is a morphism $\phi \colon (TV(B),\cdot,\bullet) \rightarrow (C,\cdot,\bullet)$ in $\HBrc$ such that $g\phi = \epsilon _B$. The fact that $g$ is central, in the sense that $[ \mathrm{Hker}(g), C ]_{\Hopfc} = \Bbbk1_{C}$, implies that $[\mathrm{Hker}(\epsilon_B) , TV(B)]_{\Hopfc} \subseteq \Hker (\phi)$. By the universal property of the quotient $\pi$ there is then a (unique) morphism $\psi \colon (A,\cdot,\bullet) \rightarrow (C,\cdot,\bullet)$ in $\HBrc$ such that $\psi\pi = \phi$. This morphism $\psi$ is such that $g\psi = f$, hence it is a morphism of (cleft) extensions and $f$ is weak $\mathcal{E}$-universal.
\end{proof}

The existence of a weak $\mathcal E$-universal central extension for any cocommutative Hopf brace $(B, \cdot, \bullet)$ will now make it possible to define and study the first and the second homology Hopf braces of $(B,\cdot,\bullet)$ relatively to the adjunction \eqref{Adj HBr Hopf}.

\emph{Baer invariants} have been widely studied in the context of semi-abelian categories with enough regular projectives \cite{EVdL,EG} and, more recently, in the category of cocommutative Hopf algebras \cite{GranSciandra-2}.  
Our approach is based on the observation that the semi-abelian category $\HBrc$ has enough $\mathcal E$-projectives (see Lemma \ref{lem:projectiveobjects}), and that a characterization of central extensions has been established in Proposition \ref{centralExt}. Note that this work can also be seen as a Hopf-theoretic extension of the recent results on the homology of skew braces achieved in \cite{GLV}.
A short exact sequence 
\begin{equation}\label{eq:shortexact}
\begin{tikzcd}
	0 & (\mathrm{Hker}(p),\cdot,\bullet) & (P,\cdot,\bullet) & (B,\cdot,\bullet) & 0
	\arrow[from=1-1, to=1-2]
	\arrow[hook, from=1-2, to=1-3, ""]
	\arrow[from=1-3, to=1-4,"p"]
	\arrow[from=1-4, to=1-5]
\end{tikzcd}
\end{equation}
in $\HBrc$ is an $\mathcal E$-\textit{projective presentation} of $(B,\cdot,\bullet)$ if $(P,\cdot,\bullet)$ is $\mathcal E$-projective and $p$ is in $\mathcal E$. 
 We will often denote an $\mathcal{E}$-projective presentation by $p:(P,\cdot,\bullet)\to(B,\cdot,\bullet)$. 

Given two $\mathcal E$-projective presentations $p:(P,\cdot,\bullet)\to(B,\cdot,\bullet)$ and $p':(P',\cdot,\bullet)\to(B',\cdot,\bullet)$, a \textit{morphism of $\mathcal{E}$-projective presentations} is a pair of morphisms $g:(P,\cdot,\bullet)\to (P',\cdot,\bullet)$, $f:(B,\cdot,\bullet)\to(B',\cdot,\bullet)$ in $\HBrc$ such that $p'g=fp$.
% \begin{equation}\label{morphismpresentations}
% \begin{tikzcd}
% 	\mathrm{Hker}(p) & P & B \\
% 	\mathrm{Hker}(p') & P' & B'
% 	\arrow[hook, from=1-1, to=1-2]
% 	\arrow[from=1-1, to=2-1, "\overline{g}"']
% 	\arrow[from=1-2, to=1-3, "p"]
% 	\arrow[from=1-2, to=2-2, "g"']
% 	\arrow[from=1-3, to=2-3, "f"]
% 	\arrow[hook, from=2-1, to=2-2]
% 	\arrow[from=2-2, to=2-3, "p'"']
% \end{tikzcd}
% \end{equation}
%where $\overline{g}$ is the unique morphism making the left-hand square commute.
We write $\mathsf{Pr}(\HBrc)$ for the category of $\mathcal{E}$-projective presentations in $\HBrc$.

A functor $F \colon  \mathsf{Pr}(\HBrc) \rightarrow \HBrc$ is called a \emph{Baer invariant} \cite{EVdL} if, given two morphisms $(g,f)$ and $(g',f')$ in $\mathsf{Pr}(\HBrc)$ from $p:(P,\cdot,\bullet)\to(B,\cdot,\bullet)$ to $p':(P',\cdot,\bullet)\to(B',\cdot,\bullet)$ as above, if 
$f=f'$ then $F((g,f))=F((g',f))$.

\begin{proposition}\label{prop: H2 Baer invariant}
    Given any $\mathcal E$-projective presentation $p:(P,\cdot,\bullet)\to(B,\cdot,\bullet)$
%\begin{equation}\label{E-proj}
%\begin{tikzcd}
	%\mathrm{Hker}(p) & P & B
	%\arrow[hook, from=1-1, to=1-2]
	%\arrow[from=1-2, to=1-3, "p"]
%\end{tikzcd}
 %\begin{tikzcd}
	%0 & R & P & B & 0
	%\arrow[from=1-1, to=1-2]
	%\arrow[" ", from=1-2, to=1-3, ""]
	%\arrow[" ",from=1-3, to=1-4," "]
	%\arrow[from=1-4, to=1-5]
%\end{tikzcd}
%\end{equation}
of a cocommutative Hopf brace $(B,\cdot,\bullet)$, the quotient  
\begin{equation}\label{Hopf}
  \mathsf{H}_2 (B) = \frac{\mathrm{Hker}(p) \cap [P,P]_{\Hopfc}}{(\mathrm{Hker}(p) \cap [P,P]_{\Hopfc})\cdot[\mathrm{Hker}(p),P]_{\Hopfc}^{+}}
\end{equation}
is a \emph{Baer invariant}, that determines a functor $\mathsf{H}_2 \colon \HBrc \rightarrow \HBrc $. 
\end{proposition}
\begin{proof}
    The arguments given in the proof of \cite[Proposition 4.1]{EG} show that the right-hand side in \eqref{Hopf} is a Baer invariant in any semi-abelian category $\mathcal A$ with enough regular projectives, where the commutators are computed relatively to any Birkhoff subcategory $\mathcal B$ of $\mathcal A$. In particular, this applies to the semi-abelian category ${\mathcal A} = \HBrc$ with respect to its Birkhoff subcategory ${\mathcal B} = \Hopfc$, with $\mathcal E$-projectives replacing regular projectives (see also \cite[Proposition $6.3$]{GranSciandra-2} for a detailed proof of a very similar result).
\end{proof}
We write $\mathsf{H}_1$ for the functor sending a cocommutative Hopf brace $(B,\cdot,\bullet)$
to its universally associated cocommutative Hopf algebra, that is to the quotient 
\[
\mathsf{H}_1(B) =F(B) = \frac{B}{B\cdot[B,B]^{+}_{\Hopfc}},
\]
where $F$ is the left adjoint in the adjunction \eqref{Adj HBr Hopf}. 
%We also write 
%for the functor sending $B$ to the quotient 
%\begin{equation}\label{Hopf-formula}
%\mathsf{H}_2 (B) =  \frac{\mathrm{Hker}(p)\cap[P,P]_{\Hopfc}}{(\mathrm{Hker}(p)\cap[P,P])[\mathrm{Hker}(p),P]_{\Hopfc}^{+}},
%\end{equation}
%where $p:P\to B$ is an arbitrary $\mathcal{E}$-projective presentation of $B$. 
By replacing regular projective presentations with $\mathcal E$-projective presentations %and short exact sequences with short $\mathcal E$-exact sequences
one can adapt the proof of \cite[Theorem 4.6]{EG} (see also \cite{GranSciandra-2}) to get the following version of the classical Stallings-Stammbach exact sequence (in group theory) for cocommutative Hopf braces:

\begin{theorem}\label{thm: Stallings-Stammbach}  
Any short exact sequence 
\begin{equation}\label{short-E-exact}
    \begin{tikzcd}
	0 & (\mathrm{Hker}(f),\cdot,\bullet) & (A,\cdot,\bullet) & (B,\cdot,\bullet) & 0
	\arrow[from=1-1, to=1-2]
	\arrow[hook, from=1-2, to=1-3]
	\arrow[from=1-3, to=1-4,"f"]
	\arrow[from=1-4, to=1-5]
\end{tikzcd}
\end{equation}
in $\HBrc$, where $f$ belongs to the class $\mathcal E$ of cleft extensions, induces the following $5$-term exact sequence in $\HBrc$
\[
\begin{tikzcd}
	 \mathsf{H}_2 (A)  & \mathsf{H}_2 (B) & \frac{\mathrm{Hker}(f)}{\mathrm{Hker}(f)\cdot[\mathrm{Hker}(f),A]_{\Hopfc}^+} & \mathsf{H}_1(A) &  \mathsf{H}_1(B) & 0.
	\arrow[ from=1-1, to=1-2, "\mathsf{H}_2(f)"]
	\arrow[from=1-2, to=1-3, ""]
    \arrow[from=1-3, to=1-4, " "]
    \arrow[from=1-4, to=1-5, "\mathsf{H}_1(f)"]
    \arrow[from=1-5, to=1-6, " "]
\end{tikzcd}
\]
\end{theorem}
\end{section}

\section{Abelian objects and another central series}\label{abelian}

% In Section \ref{sec: relative commutator} we have investigated central extensions and commutators relative to the adjunction \eqref{Adj HBr Hopf}. We now want to compare these notions with central extensions and commutators defined in \cite{GranSciandra} and link them to another lower central series. 

% We recall that, in any semi-abelian category $\mathcal{C}$, the subcategory $\mathsf{Ab}(\mathcal{C})$ of abelian objects in $\mathcal{C}$ is a Birkhoff subcategory of $\mathcal{C}$.
It is shown in \cite[Proposition 5.5]{GranSciandra} that the Birkhoff subcategory of abelian objects of $\HBrc$ coincides with the category $\Hopfcc$, and the abelianisation functor $\mathsf{ab}$
%the two Hopf algebra structures coincide (i.e.\ $\cdot=\bullet$) and, moreover, they are abelian (i.e.\ commutative and cocommutative Hopf algebras). 
% Therefore, we have the following adjunction \cite[Corollary 8.8]{GranSciandra}:
% \begin{equation}
% \begin{tikzcd}
% \HBrc &&& \mathsf{Hopf}_{\mathrm{coc}}^{\mathsf{com}}
% 	\arrow[" "{name=1, anchor=center, inner sep=0}, curve={height=-8pt}, from=1-1, to=1-4,"\mathsf{ab}"]
% 	\arrow[" "{name=1, anchor=center, inner sep=0}, curve={height=-8pt}, from=1-4, to=1-1, "U"]
%     \arrow["\vdash"{anchor=center, rotate=90}, draw=none, from=1-4, to=1-1].
% \end{tikzcd}
% \end{equation}
% \noindent where $U$ is the functor sending any commutative and cocommutative Hopf algebra to the cocommutative Hopf brace whose two multiplications coincide, and %The explicit description of the left adjoint $\mathsf{ab}$ is in \cite[Corollary 8.8]{GranSciandra}. It 
in \eqref{adj Hbr Hopfcc} is defined, on any Hopf brace $(H,\cdot,\bullet)$,  by the assignment
$$(H,\cdot,\bullet)\longmapsto \mathsf{ab}(H)= \frac{H}{H\cdot [H,H]^+},$$
where $[H,H]$ is the Hopf subalgebra of $H^{\cdot}$---that is in fact a normal Hopf subbrace---generated by
$\{a\star b,[a,b]_\cdot\mid a,b\in H\},$
where $[a,b]_\cdot=a_1\cdot b_1\cdot S(a_2)\cdot S(b_2)$. 
One can compute the functor $\mathsf{ab} \colon \HBrc \rightarrow \Hopfcc$ as the composite of $F\colon \HBrc\to \Hopfc$ with the left adjoint $\mathsf{ab}'\colon \Hopfc \to \Hopfcc$ to the forgetful functor $U'\colon\Hopfcc\to\Hopfc$, since \eqref{adj Hbr Hopfcc} is the composite of the following two adjunctions
\begin{equation}
    \begin{tikzcd}
\HBrc &&& \Hopfc &&& \Hopfcc
	\arrow[" "{name=1, anchor=center, inner sep=0}, curve={height=-8pt}, from=1-1, to=1-4,"F"]
	\arrow[" "{name=1, anchor=center, inner sep=0}, curve={height=-8pt}, from=1-4, to=1-1, "G"]
    \arrow["\vdash"{anchor=center, rotate=90}, draw=none, from=1-4, to=1-1]
    \arrow[" "{name=1, anchor=center, inner sep=0}, curve={height=-8pt}, from=1-4, to=1-7,"\mathsf{ab'}"]
	\arrow[" "{name=1, anchor=center, inner sep=0}, curve={height=-8pt}, from=1-7, to=1-4, "U'"]
    \arrow["\vdash"{anchor=center, rotate=90}, draw=none, from=1-7, to=1-4].
\end{tikzcd}
\end{equation}
For a cocommutative Hopf brace $(H,\cdot,\bullet)$ and a normal Hopf subbrace $(I,\cdot,\bullet)$, the commutator relative to \eqref{adj Hbr Hopfcc} coincides with the Huq commutator described in \cite{GranSciandra}:
$$[I,H]_\Hopfcc=[I,H]_\mathsf{Huq}=\langle\{[i,h]_\cdot,[i,h]_\bullet,i\star h\mid i\in I, h\in H\}\rangle_N,$$
where $\langle-\rangle_N$ denotes the normal closure in $H$.
By applying the results in \cite{BournGran} to the semi-abelian category $\HBrc$, we know that a surjective morphism $f\colon (A,\cdot,\bullet)\to(B,\cdot,\bullet)$ in $\HBrc$ is a central extension with respect to the adjunction \eqref{adj Hbr Hopfcc} if and only if $\Hker(f)$ and $A$ commute in the sense of Huq, namely if and only if $[\Hker(f),A]_\text{Huq}=\Bbbk1_{A}$. By \cite[Lemma 8.2]{GranSciandra}, this is also equivalent to the fact that $a\cdot k=k\cdot a=k\bullet a =a\bullet k$ for any $k\in \Hker(f)$ and any $a\in A$.

We now define another central series that can be interpreted as a generalization of the lower central series of skew braces defined in \cite{BP} (see also \cite{KanrarRoelantsYadav}). Given a cocommutative Hopf brace $(H,\cdot,\bullet)$, we define
\begin{equation*}
    \Gamma_1(H)\coloneqq H,\quad \quad%\Gamma_{n+1}(H)\coloneqq \langle 
    %\Gamma_n(H)\star H,H\star \Gamma_n(H),[H,\Gamma_n(H)]_\cdot\rangle_\cdot,\quad 
    \Gamma_{n+1}(H)\coloneqq \langle\{ 
    i\star h,h\star i,[h,i]_\cdot\ |\ h\in H,\ i\in \Gamma_n(H)\}\rangle_\cdot
\end{equation*}
where %$[H,\Gamma_n(H)]_\cdot=\{h_1\cdot a_1\cdot S(h_2)\cdot S(a_2)\mid h\in H, a\in\Gamma_n(H)\}$ and 
$\langle-\rangle_\cdot$ denotes the generated Hopf subalgebra of $H^{\cdot}$.

\begin{proposition}\label{prop: Gamma_n normal}
    Let $(H,\cdot,\bullet)$ be a cocommutative Hopf brace. For any $n\geq 1$, $\Gamma_n(H)$ is a normal Hopf subbrace of $(H,\cdot,\bullet)$.
\end{proposition}
\begin{proof}
We prove the thesis by induction on $n$. For $n=1$, the statement is trivial. Now, we assume that $\Gamma_n(H)$ is a normal Hopf subbrace of $(H,\cdot,\bullet)$ for an arbitrary $n$ and we prove that $\Gamma_{n+1}(H)$ is normal as well. We first observe that, since $i\star h,h\star i\in\Gamma_{n}(H)$ for any $i\in\Gamma_{n}(H)$ and $h\in H$ by Corollary \ref{cor:equivalentnotionsnormal} and $[h,i]_{\cdot}\in\Gamma_{n}(H)$ as well, we have $\Gamma_{n+1}(H)\subseteq\Gamma_{n}(H)$.

\noindent 1) Firstly, we check that $H\rightharpoonup \Gamma_{n+1}(H)\subseteq \Gamma_{n+1}(H)$. It suffices to check it for the generating elements of $\Gamma_{n+1}(H)$. Using 4) of Lemma \ref{lem:propertiesasterisk}, for any $x,y\in H$ and $a\in \Gamma_n(H)$, we get:
        \begin{align*}
           x\rightharpoonup(a\star y)&=%\overset{\eqref{lem:propertiesasterisk}}{=} 
           (x_1\bullet a\bullet T(x_2))\star (x_3\rightharpoonup y)\in \Gamma_n(H)\star H\subseteq\Gamma_{n+1}(H)\\
           x\rightharpoonup (y\star a)&=%\overset{\eqref{lem:propertiesasterisk}}{=}
           (x_1\bullet y\bullet T(x_2))\star (x_3\rightharpoonup a)\in H\star \Gamma_n(H)\subseteq\Gamma_{n+1}(H)\\
           x\rightharpoonup [y,a]_\cdot&=(x_{1}\rightharpoonup y_{1})\cdot(x_{2}\rightharpoonup a_{1})\cdot(x_{3}\rightharpoonup S(y_{2}))\cdot(x_{4}\rightharpoonup S(a_{2}))\\&=(x_{1}\rightharpoonup y_{1})\cdot(x_{3}\rightharpoonup a_{1})\cdot S(x_{2}\rightharpoonup y_{2})\cdot S(x_{4}\rightharpoonup a_{2})\\&=[x_1\rightharpoonup y,x_2\rightharpoonup a]_{\cdot}\in [H,\Gamma_n(H)]_\cdot\subseteq\Gamma_{n+1}(H),
        \end{align*}
        as $\Gamma_n(H)$ is normal in $H^{\bullet}$ and $H\rightharpoonup \Gamma_n(H)\subseteq \Gamma_n(H)$.

\noindent 2) For any $x,y\in H$ and $a\in \Gamma_n(H)$, using 1) of Lemma \ref{lem:propertiesasterisk}, we obtain
        \begin{align*}
            x_1\cdot (a\star y)\cdot S(x_2)&=S(a_1\star x_1)\cdot (a_2\star x_2)\cdot x_3\cdot (a_3\star y)\cdot S(x_4)\\
            &=S(a_1\star x_1)\cdot (a_2\star(x_2\cdot y))\in \Gamma_n(H)\star H\subseteq\Gamma_{n+1}(H).
        \end{align*}
        \noindent Furthermore, since $y\star a\in \Gamma_n(H)$ %for all $a\in\Gamma_{n}(H),y\in H$ 
        (Corollary \ref{cor:equivalentnotionsnormal}) and $\Gamma_n(H)$ is normal in $H^{\cdot}$, we have
        \begin{align*}
            x_1\cdot (y\star a)\cdot S(x_2)&=[x_1,y_1\star a_1]_\cdot \cdot (y_2\star a_2)\in [H,\Gamma_n(H)]_\cdot\cdot (H\star \Gamma_n(H))\subseteq\Gamma_{n+1}(H)\\
            x_1\cdot [y,a]_{\cdot}\cdot S(x_2)&=[x_1\cdot y\cdot S(x_2),x_3\cdot a\cdot S(x_4)]_\cdot \in [H,\Gamma_n(H)]_\cdot\subseteq\Gamma_{n+1}(H).
        \end{align*}
        Thus, we obtain that $\Gamma_{n+1}(H)$ is normal in $H^{\cdot}$.

\noindent 3) Finally, we prove that $\Gamma_{n+1}(H)$ is normal in $H^{\bullet}$. Given $a\in\Gamma_{n+1}(H)$ and $x\in H$, since $\Gamma_{n+1}(H)\subseteq\Gamma_{n}(H)$, we have $a_1\cdot (a_2\star T(x))\in \Gamma_{n+1}(H)\cdot (\Gamma_{n}(H)\star H)\subseteq \Gamma_{n+1}(H),$
        so that, by 5) of Lemma \ref{lem:propertiesasterisk} and points 1) and 2) proven above, we obtain:
        \[
        x_1\bullet a\bullet T(x_2)=x_1\cdot (x_2\rightharpoonup (a_1\cdot (a_2\star T(x_3))))\cdot S(x_4)\in \Gamma_{n+1}(H).\qedhere\]
%This concludes the proof that $\Gamma_{n+1}(H)$ is a normal Hopf subbrace of $H$.
\end{proof}

We can now show the link between the series just defined and the Huq commutator, that is also the commutator relative to the adjunction \eqref{adj Hbr Hopfcc}.

\begin{proposition}
    For any $n\geq 1$, the equality $\Gamma_{n+1}(H)=[\Gamma_n(H),H]_\mathsf{Huq}$ holds.
\end{proposition}
\begin{proof}
    The inclusion from left to right follows from the fact that the Huq commutator is symmetric, namely, for any two normal Hopf subbraces $(I,\cdot,\bullet)$ and $(J,\cdot,\bullet)$ of $(H,\cdot,\bullet)$, we have $[I,J]_\mathsf{Huq}=[J,I]_\mathsf{Huq}$. Therefore, we have that
    \begin{align*}
    H\star \Gamma_n(H),[H,\Gamma_n(H)]_\cdot&\subseteq [H,\Gamma_n(H)]_\mathsf{Huq}=[\Gamma_n(H),H]_\mathsf{Huq}\supseteq\Gamma_{n}(H)\star H %\\
%[H,\Gamma_n(H)]_\cdot&\subseteq [H,\Gamma_n(H)]_\mathsf{Huq}=[\Gamma_n(H),H]_\mathsf{Huq} 
    \end{align*}
    which implies $\Gamma_{n+1}(H)\subseteq [\Gamma_n(H),H]_\mathsf{Huq}$. 
     
    To prove the converse inclusion $[\Gamma_n(H),H]_\mathsf{Huq}\subseteq \Gamma_{n+1}(H)$, recall that $\Gamma_{n+1}(H)$ is a normal Hopf subbrace of $(H,\cdot,\bullet)$ (Proposition \ref{prop: Gamma_n normal}) containing $i\star h$ and $[i,h]_\cdot=S([h,i]_\cdot)$ for all $i\in\Gamma_n(H)$ and $h\in H$. Hence it suffices to check that $[i,h]_\bullet\in\Gamma_{n+1}(H)$ for all $i\in\Gamma_n(H)$ and $h\in H$. 
    Let $\pi\colon(H,\cdot,\bullet)\to(\overline{H},\cdot,\bullet)$ be the quotient of $(H,\cdot,\bullet)$ by the normal Hopf subbrace $\Gamma_{n+1}(H)$, so that $\Hker(\pi)=\Gamma_{n+1}(H)$ by \cite[Proposition 4.6]{GranSciandra}, and write $\overline{h}:=\pi(h)$ for $h\in H$. Recall that $\pi(k)=\varepsilon(k)1$ for any $k\in\Hker(\pi)$. Let $i\in\Gamma_n(H)$ and $h\in H$. Since $\Gamma_n(H)$ is a subcoalgebra of $H$, the elements $i_{j}$ appearing in the iterated comultiplications of $i$ belong to $\Gamma_n(H)$, hence the elements $i_j\star h_k$, $h_k\star i_j$ and $[h_k,i_j]_\cdot$ lie in $\Gamma_{n+1}(H)=\Hker(\pi)$ and are sent by $\pi$ to the corresponding multiples of $1$. Since $\pi$ is a morphism of Hopf braces, we then get
    \[
    \overline{i}\bullet\overline{h}=\overline{i}_1\cdot(\overline{i}_2\star\overline{h}_1)\cdot\overline{h}_2=\overline{i}\cdot\overline{h},\qquad
    \overline{h}\bullet\overline{i}=\overline{h}_1\cdot(\overline{h}_2\star\overline{i}_1)\cdot\overline{i}_2=\overline{h}\cdot\overline{i},
    \]
    and, using cocommutativity,
    \[
    \overline{h}\cdot\overline{i}=\overline{h}_1\cdot\overline{i}_1\cdot S(\overline{h}_2)\cdot S(\overline{i}_2)\cdot\overline{i}_3\cdot\overline{h}_3=[\overline{h}_1,\overline{i}_1]_\cdot\cdot\overline{i}_2\cdot\overline{h}_2=\overline{i}\cdot\overline{h}.
    \]
    It follows that $\overline{i}\bullet\overline{h}=\overline{h}\bullet\overline{i}$, and the same holds for the Sweedler components of $i$ and $h$. Therefore
    \[
    \pi([i,h]_\bullet)=\overline{i}_1\bullet\overline{h}_1\bullet T(\overline{i}_2)\bullet T(\overline{h}_2)=\overline{h}_1\bullet\overline{i}_1\bullet T(\overline{i}_2)\bullet T(\overline{h}_2)=\varepsilon(i)\varepsilon(h)1.
    \]
  Finally, since $\Delta([i,h]_\bullet)=[i_1,h_1]_\bullet\ot[i_2,h_2]_\bullet$ with $i_2\in\Gamma_n(H)$, the previous computation also applies to $[i_2,h_2]_\bullet$, and we get
    \[
    ([i,h]_\bullet)_1\ot\pi\big(([i,h]_\bullet)_2\big)=[i_1,h_1]_\bullet\ot\varepsilon(i_2)\varepsilon(h_2)1=[i,h]_\bullet\ot1,
    \]
    i.e.\ $[i,h]_\bullet\in\Hker(\pi)=\Gamma_{n+1}(H)$.

\end{proof}

  %  To prove the converse inclusion $[H,\Gamma_n(H)]_\mathsf{Huq}\subseteq \Gamma_{n+1}(H)$, we just have to verify that $[H,\Gamma_n(H)]_\bullet\subseteq\Gamma_{n+1}(H)$. But, for any $x\in H$ and $a\in \Gamma_n(H)$, we have:
 %   $$x_1\bullet a_1\bullet T(x_2)\bullet T(a_2)\in \Gamma_{n+1}(H)\bullet \Gamma_{n+1}(H)=\Gamma_{n+1}(H)$$
%    because $\Gamma_{n+1}(H)$ is normal by Proposition \ref{prop: Gamma_n normal}.    

Proceeding along the same lines as in Section \ref{section: Weak univ centr ext}, one can prove, for any cocommutative Hopf brace $(H,\cdot,\bullet)$, the existence of a weak $\mathcal{E}$-universal central extension with respect to the adjunction \eqref{adj Hbr Hopfcc}. The construction is similar to the one presented in Proposition \ref{thm: weak E-univ centr extension}, but replacing the commutator relative to \eqref{Adj HBr Hopf} with the Huq commutator: one considers the cokernel $\pi\colon (TV(B),\cdot,\bullet)\to (A,\cdot,\bullet)$ of the inclusion $[\Hker (\epsilon_B),TV(B)]_\mathrm{Huq}\hookrightarrow TV(B)$ in $\HBrc$. The regular epimorphism $\pi$ induces a unique morphism $f\colon (A,\cdot,\bullet)\to(B,\cdot,\bullet)$ that is a weak $\mathcal{E}$-universal central extension with respect to \eqref{adj Hbr Hopfcc}. %(i.e. with respect to the adjunction \ref{adj Hbr Hopfcc}).

Furthermore, the arguments presented in Section \ref{section: Weak univ centr ext} can also be adapted to the adjunction \eqref{adj Hbr Hopfcc}. For any cocommutative Hopf brace $(B,\cdot,\bullet)$, we can define:
\[
\mathsf{H}_1^{\mathsf{ab}}(B)=\mathsf{ab}(B)\quad\quad \mathsf{H}_2^{\mathsf{ab}}(B)= \frac{\mathrm{Hker}(p)\cap[P,P]_{\mathsf{Huq}}}{(\mathrm{Hker}(p)\cap[P,P]_{\mathsf{Huq}})\cdot[\mathrm{Hker}(p),P]_\mathsf{Huq}^{+}},\]
where $p:(P,\cdot,\bullet)\to(B,\cdot,\bullet)$ is an arbitrary $\mathcal{E}$-projective presentation of $(B,\cdot,\bullet)$. %(the definition of $\mathsf{H}_2(B)$ does not depend on the choice of $p$ for the same argument as Proposition \ref{prop: H2 Baer invariant}). 
{Starting from any short exact sequence as in \eqref{short-E-exact}, one can derive another version of the Stallings-Stammbach exact sequence:
\[
\begin{tikzcd}
	\mathsf{H}_2^{\mathsf{ab}} (A)  & \mathsf{H}_2^{\mathsf{ab}} (B) & \frac{\mathrm{Hker}(f)}{\mathrm{Hker}(f)\cdot[\mathrm{Hker}(f), A]_{\mathsf{Huq}}^+} & \mathsf{H}_1^{\mathsf{ab}}(A) &  \mathsf{H}_1^{\mathsf{ab}}(B) & 0.
	\arrow[ from=1-1, to=1-2, "\mathsf{H}_2^{\mathsf{ab}}(f)"]
	\arrow[from=1-2, to=1-3, ""]
    \arrow[from=1-3, to=1-4, " "]
    \arrow[from=1-4, to=1-5, "\mathsf{H}_1^{\mathsf{ab}}(f)"]
    \arrow[from=1-5, to=1-6, " "]
\end{tikzcd}
\]
\noindent\textbf{Acknowledgments}. M. Bevilacqua's research is funded by a FRIA doctoral grant from the Fonds de la Recherche Scientifique - FNRS. M. Gran’s research was supported by the Fonds de la Recherche Scientifique - FNRS under Grant CDR No.J.0092.26. A. Sciandra was supported by a postdoctoral fellowship at the ULB in the framework of the PDR project ``Reconstruction of modules and algebraic objects from closed and monoidal structures on their representation categories'' funded by the FNRS under the grant number T.0318.25F (PI J. Vercruysse). This paper was written while he was member of the GNSAGA-INdAM.

\bibliographystyle{amsplain}

%    Bibliographies can be prepared with BibTeX using amsplain,
%    amsalpha, or (for "historical" overviews) natbib style.

\end{document}